\documentclass{tac}

\usepackage{amsmath,amsfonts,amssymb}
\usepackage{latexsym}
\usepackage{comment}
\usepackage{color}
\usepackage{verbatim}
\usepackage{esint}
\usepackage{amscd}
\usepackage{euscript}
\usepackage{mathtools}
\usepackage{graphicx}
\usepackage{hyperref}
\usepackage{graphicx}
\usepackage{rotating}
\usepackage[usenames,dvipsnames]{xcolor}
\usepackage{fancybox}
\usepackage{enumitem}
\usepackage{stmaryrd}
\usepackage{hyperref}
\usepackage{pifont}
\usepackage{relsize}
\usepackage{lipsum}
\usepackage{tikz}
\usepackage{lmodern,graphicx}
\usepackage{lmodern}
\usepackage{cleveref}
\usepackage[all]{xy}
\usepackage[square]{natbib}

\input diagxy

\usepackage[colorlinks=true]{hyperref}
\hypersetup{allcolors=[rgb]{0.1,0.1,0.4}}

\author{Kad\.{i}r Em\.{i}r}

\address{Department of Mathematics and Statistics, \\ Faculty of Science, \\ Masaryk University, \\ Brno, Czech Republic. \\}

\title {The Moore Complex of a Simplicial Cocommutative Hopf Algebra}

\copyrightyear{2021}

\keywords{Hopf algebra, simplicial object, Moore complex, 2-crossed module}
\amsclass{16T05, 16S40, 18G45, 55U10, 55U15}

\eaddress{emir@math.muni.cz}

\newtheorem{thm}{Theorem}
\newtheorem{theorem}{Theorem} 
\newtheorem{lem}{Lemma}
\newtheorem{deco}{D\'ecalage Construction} 
\newtheorem{decomp}{Decomposing}

\newtheoremrm{rem}{Remark}
\newtheoremrm{idea}{Idea}

\mathrmdef{Hom}
\mathbfdef{Set}

\def \aa {\vartriangleright_{ad}}
\def \arho {\vartriangleright_{\rho}}
\def \e {\epsilon}
\def \n {\noindent}
\def \C {\kappa}

\def \act {\vartriangleright}

\def \ad {\vartriangleright_{ad}}

\def \q {\quad}

\def \tr {\triangleright_\rho}
\def \tn {\otimes}

\def \id {{\rm id}}

\def \lk {{{\rm HKer}}}

\def \D {\Delta}
\def \e {\epsilon}

\def \k {\kappa}

\def \tr {\triangleright}

\def \d {\partial}

\def \F {\mathcal{F}}

\def \H {\mathcal{H}}

\def \n {\noindent}

\def \rh {\vartriangleright_{\rho}}
\def \rhp {\vartriangleright'_{\rho}}

\def \qqqq {\quad\quad\quad\quad\quad\quad}

\def \ra {\xrightarrow}

\def \edit {\textcolor{black}}
\def \rev {\textcolor{black}}
\def \revt {\textcolor{black}}

\def \seven {\big( l \otimes m \otimes z \otimes k \otimes x \otimes y \otimes a \big)}

\makeatother
\newcommand{\tsma}[1] {\vartriangleright_{\rho_{#1}}}
\newcommand{\osma}[1] {\otimes_{\rho_{#1}}}
\makeatletter

\allowdisplaybreaks

\begin{document}

\maketitle

\begin{abstract}
	We \rev{study} the Moore complex of a simplicial cocommutative Hopf algebra through Hopf kernels. The most striking result to emerge from this construction is the coherent definition of 2-crossed modules of cocommutative Hopf algebras. This unifies the 2-crossed module theory of groups and of Lie algebras when we take the group-like and primitive functors into consideration.
\end{abstract}

% NOTE: it is good practice to \label all headings (and proclamations) immediately

\tableofcontents

\section{Introduction}

A simplicial group $\mathcal{G}=(G_n,d_i^n,s_j^{n+1})$ is a simplicial object \citep{PM1} in the category of groups. It is given by a collection of groups $G_n$, together with group homomorphisms $d_i^n \colon G_n \to G_{n-1}$ and $s_j^{n+1}\colon G_n \to G_{n+1}$ for $i,j=0, \dots,n$ called faces and degeneracies, respectively, satisfying the well known simplicial identities. The Moore complex \mbox{\citep{Moore1}} of a simplicial group is a chain complex
\begin{align*}
N(\mathcal{G})= \big( \dots \ra{\d_{(n+1)}}  N(G)_n \ra{\d_n} \dots  \ra{\d_{3}} N(G)_2 \ra{\d_2} N(G)_1 \ra{\d_1} G_0  \big)
\end{align*}
of groups, where $$N(G)_n=\bigcap_{i=0}^{n-1} \ker(d_i^n)$$ at level $n$, and the boundary morphisms $\d_n\colon N(G)_n \to N(G)_{n-1}$ are the restrictions of the $d_n^n \colon G_n \to G_{n-1}$. Moreover, $N(\mathcal{G})$ defines a normal chain complex of groups, namely $\d_n(N(G)_n) \trianglelefteq N(G)_{n-1}$, for all $n \geq 1$. Thus, the Moore complex can be considered as the normalized chain complex of a simplicial group. \rev{This construction was first considered as a Moore complex functor from the category of simplicial abelian groups to the category of chain complexes of abelian groups - that yields an equivalence between these two categories (called Dold-Kan correspondence) which was independently proven by \citep{10.2307/1970043,zbMATH03148099}. For more details, see \citep{GJ1}. Afterwards, the Dold-Kan correspondence was studied for various cases, and consequently generalized to abelian categories in \citep{zbMATH03160777}, to semi-abelian categories in \citep{zbMATH05219531}, and to more general source categories and settings rather than simplicial objects in \citep{zbMATH06444745}. Above all, we focus our attention on \citep{zbMATH05219531} which considers the Moore complex structure in semi-abelian categories, of which the category of cocommutative Hopf algebras are an instance -- see \citep{zbMATH06579662, zbMATH06970019}, also \citep{vespa} and \citep{zbMATH07075977}. Bourn's work builds on \citep{zbMATH02125642}, and semi-abelian categories were introduced in \citep{zbMATH01724907}. In the interests of brevity, a semi-abelian category is a category with binary coproducts which is pointed, Barr exact and Bourn protomodular. And the category of cocommutative Hopf algebras over any field is semi-abelian, as well as the category of groups and Lie algebras.}

\medskip

\rev{Not only in the category-theoretic sense, but the Moore complex also has many roles in the context of algebraic topology. For instance, it is a well-known property that the $n^{th}$ homotopy group of a simplicial group $\mathcal{G}$ is the $n^{th}$ homology group of the Moore complex $N(\mathcal{G})$. Furthermore, as an extension of the Dold-Kan correspondence to the category of arbitrary groups (i.e. not only abelian groups), it is shown in \citep{CC1} that the Moore complex construction yields a functor from the category of simplicial groups to the category of hypercrossed complexes of groups. A recent study from the same authors also deals with the Lie algebraic case of the same problem \citep{zbMATH06785333}. A hypercrossed complex is a chain complex of groups that comes equipped with a specific type of binary operations satisfying certain axioms. Hypercrossed complexes have the following essential property, which is strongly related to this study: they capture crossed modules and 2-crossed modules for dimensions one and two, respectively.}

\medskip

A crossed module is a group homomorphism $\d \colon E \to G$, together with an action $\vartriangleright$ of $G$ on $E$ \rev{by automorphisms} satisfying \revt{$ \d(g \vartriangleright e) =g \,\d(e) \, g^{-1}$ (equivariance) and $ \d(e)  \vartriangleright f =e f  e^{-1} $ (Peiffer condition), for all $e,f \in E$ and $g \in G$}. The notion is introduced in \citep{WCH2} as an algebraic model of connected homotopy 2-types. From another point of view, a crossed module can be considered as an encoded strict 2-group \citep{zbMATH03521203}. \rev{Categorically, crossed modules are equivalent to cat$^1$-groups, which can be shown to be equivalent to internal categories in the category of groups \citep{Lo}}. For more details, see \citep{zbMATH03787031, zbMATH03989648, Br, BL, zbMATH05815821, zbMATH06520714}. Also a thorough discussion of crossed modules from the topological and algebraic point of view is given in \mbox{\citep{zbMATH06825034}}.

\medskip

For a given simplicial group $\mathcal{G}$, we say that the Moore complex $N(\mathcal{G})$ is of \mbox{length $n$}, if $N(G)_i$ is trivial for all $i > n$. In the case that $n=1$ it yields a crossed module  $N(G)_1 \ra{\d_1} G_0 $, where the action of $G_0$ on $N(G)_1$ is defined using the conjugate action via $s_0$. Inspired by this functorial relationship between simplicial groups and crossed modules, Conduch\'e introduced 2-crossed modules of groups in \citep{C2}. Namely, for a given simplicial group $\mathcal{G}$ with Moore complex of length two, it is shown that the first three \rev{levels} of the Moore complex $N(G)_2 \overset{\partial _{2}}{\longrightarrow } N(G)_1 \overset{\partial _{1}}{	\longrightarrow } G_0$ leads to the 2-crossed module definition. \mbox{A 2-crossed} module of groups is a complex $L \to E \to G$ of groups together with left actions $\vartriangleright$ of $G$ on $L,E$, and on itself by conjugation; and a $G$-equivariant function $\{\, ,\} \colon E \times E \to L$ called Peiffer lifting, satisfying certain properties. An alternative way to obtain a 2-crossed module from a simplicial group without considering the length of the Moore complex is given in \citep{MP2}. Another analogy from crossed modules is that 2-crossed modules are also algebraic models for connected homotopy 3-types, that is, pointed CW-complexes X such that $\pi_i (X)=0$ for $i \geq 3$ \citep{CC1}. Additionally, there are some other algebraic models of homotopy 3-types such as braided crossed modules \citep{zbMATH04051863}, neat crossed squares \citep{JFM0} and Gray 3-groups \citep{zbMATH01786898} (these three are equivalent to the 2-crossed modules); crossed squares \citep{zbMATH00561288} and quadratic modules \citep{Baues} (being the particular case of 2-crossed modules).  Furthermore, as a generalization, 2-quasi crossed modules are introduced \citep{CP1} in which some conditions are relaxed.

\medskip

Regarding to all group theoretic terminology given above, and in the light of the close analogy between general algebraic properties of groups and Lie algebras; the Lie algebraic case of the whole 2-crossed module theory is given in \citep{E1}, based on \citep{zbMATH03762250} in which the crossed modules are introduced in the context of Lie algebras. However, the analogy between groups and Lie algebras becomes more powerful in the category of Hopf algebras that allows us to unify both of these group and Lie algebraic theories in a functorial way, which was the main motivation of this study.

\medskip

Hopf algebras \citep{SW1} can be thought \rev{of} as a unification of groups and Lie algebras as being the group algebra of a group \rev{via the functor $\cal G^\sharp$}, and the universal enveloping algebra of a Lie algebra \rev{via the functor $\cal F^\sharp$}.\footnote{\revt{The functors denoted by $\cal F^{\sharp}$ and $\cal G^{\sharp}$ are usually written as $U$ and $\kappa [\cdot]$, respectively, where $\kappa$ is the field on which the Hopf algebras are built.}} \rev{Conversely,} we have the functors $Gl$ and $Prim$ from the category of Hopf algebras to the category of groups and of Lie algebras which assigned group-like and primitive  elements, respectively \rev{(for full details, we refer to \citep{zbMATH06579662}[\S 2.2])}. 
%\rev{Such functors are given in the following picture including their adjunctions:\footnote{We denote the categories of \textit{all} Hopf algebras by $\mathrm{Hopf}$, and \textit{cocommutative} Hopf algebras by $\mathrm{Hopf^{cc}}$ throughout the paper.}}
%\begin{align}\label{adjuntion}
%	\rev{ 	
%			\xymatrix@C=60pt@R=70pt{
%			\mathrm{Lie} 
%			\ar@<-1.1ex>@{}[r]^(.13){}="a"^(.8){}="b" \ar_{\mathcal{U}} "a";"b"="9"
%			& \mathrm{Hopf} 
%			\ar@<-1.1ex>@{}[l]^(.2){}="a"^(.87){}="b" \ar_{Prim} "a";"b"="10" 
%			\ar@{}"9";"10"|(.2){\,}="11"
%			\ar@{}"9";"10"|(.8){\,}="12"
%			\ar@{}"12" ;"11"|{\top}
%			\ar@<-1.1ex>@{}[r]^(.18){}="a"^(.85){}="b" \ar_{Gl} "a";"b"="14"
%			&  \mathrm{Grp}
%			\ar@<-1.1ex>@{}[l]^(.16){}="a"^(.82){}="b" \ar_{\k[-]} "a";"b"="13" 
%			\ar@{}"13";"14"|(.2){\,}="15"
%			\ar@{}"13";"14"|(.8){\,}="16"
%			\ar@{}"15" ;"16"|{\bot}
%			 }
%	}
%\end{align}
\rev{It is crucial that both the group algebra and the universal enveloping algebra turns into a specific type of Hopf algebras called \textit{cocommutative}. Following from the cocommutativity property, we also have a coreflection functor $ \revt{R \colon} \mathrm{\{Hopf\}} \to \mathrm{\{Hopf^{cc}\}}$ \citep{zbMATH05321429} from the category of \textit{all} Hopf algebras to the category of \textit{cocommutative} Hopf algebras, the latter being a full and replete subcategory which is also semi-abelian. There exists a category $\mathrm{\{Grp\ltimes Lie\}}$ whose objects are triples $(G,L,\tr)$, where $G$ is a group, $L$ is a Lie algebra and $\tr$ is a representation of $G$ on $L$ by Lie algebra maps. We have a functor ${ \cal G^\sharp \ltimes \F^\sharp}\colon \mathrm{\{Grp\ltimes Lie \}} \to  \mathrm{\{Hopf^{cc}\}}$ sending $(G,L,\tr)$ to \mbox{${\cal F}^\sharp(L)\otimes_{\rho} {\cal G}^\sharp(G)$} (for the notation, see Example \ref{Cartier}). Conversely, we naturally have ${ Gl\ltimes Prim } \colon $ $\mathrm{\{Hopf^{cc}\} \to \{Grp\ltimes Lie\}}$ noting that the group of group-like elements acts on the set of primitive elements by conjugation. In particular, if the base field of a cocommutative Hopf algebra is algebraically closed and of characteristic zero, then the functors ${ \cal G^\sharp \ltimes \F^\sharp}$ and ${ Gl\ltimes Prim }$ define an equivalence of categories  $\mathrm{\{Hopf^{cc}\}\cong \{Grp\ltimes Lie\}}$. This is one of the ways to state the \mbox{Cartier-Gabriel-Kostant-Milnor-Moore} theorem.} 
%\rev{Except the cocommutative ones, there also exist some other} well-known variations of Hopf algebras in the literature that are obtained by either relaxing some properties or adding extra structure. For instance, quasi-Hopf algebras \citep{bulacu, quasi}, quasi-triangular Hopf algebras \citep{SM1}, quantum groups \citep{zbMATH05015420,1901.04328},  Leibniz-Hopf algebras \citep{hazen}, Steenrod algebras \citep{steenrod}, Hopfish algebras \citep{hopfish}.  
\rev{The full diagram of these categories and functors is:}
\begin{equation}\label{bigone}
\xymatrix@C=45pt@R=40pt{     
	\mathrm{\{Hopf\}}  \ar@<1.2ex>[r]^{R }_{\top}
	& \mathrm{\{Hopf^{cc}\}} \ar@<0.7ex>[l]^{\iota} 
	\ar@/^1.0pc/@{->}[rr]_{ Gl\ltimes Prim }="2"
	&  & \mathrm{\{Grp\ltimes Lie\}} 
	\ar@/_2.2pc/[]!<-2ex,0ex>;[ll]!<0ex,0ex>_{ \cal G^\sharp \ltimes \F^\sharp}="1"
	\ar@{}"1";"2"|(.2){\,}="3"
	\ar@{}"1";"2"|(.8){\,}="4"
	\ar@{}"3" ;"4"|{\bot }
	& 
	\\
	&& \mathrm{\{Grp\}}
	\ar@<-1.0ex>@{}[r]^(.15){}="a"^(.87){}="b" \ar_{U} "a";"b"="9" 
	\ar@/_1.0pc/@{<-}[]!<-1.5ex,2ex>;[lu]!<0ex,0ex>_-{Gl}="1" 
	\ar@/^1.0pc/[]!<0ex,0.5ex>;[lu]!<2ex,0ex>^{\cal G^\sharp }="2"
	\ar@{}"1";"2"|(.5){\,}="3"
	\ar@{}"1";"2"|(.5){\,}="4"
	\ar@{}"3" ;"4"|{\scalebox{0.8}{\rotatebox{-45}{$\top$}}}
	\ar@/^1.5pc/@{->}[]!<0ex,0ex>;[ru]!<0ex,0ex>^(.55){\iota_1}="21"
	\ar@/_1pc/@{<-}[]!<-2.5ex,0ex>;[ru]!<-1ex,3ex>_(.6){U_1}="22"
	\ar@{}"21";"22"|(.17){\,}="23"
	\ar@{}"21";"22"|(.83){\,}="24"
	\ar@{}"23" ;"24"|{\scalebox{0.8}{\rotatebox{225}{$\top$}}}
	&\mathrm{\{Set\}}
	\ar@<-0.75ex>@{}[l]^(.13){}="a"^(.85){}="b" \ar_{\cal F} "a";"b"="10"
	\ar@{}"9";"10"|(.2){\,}="11"
	\ar@{}"9";"10"|(.8){\,}="12"
	\ar@{}"12" ;"11"|{\bot}  
	% 	\ar@{<-}[u]<1.3ex>^{U}_{\vdash}
	% 	\ar[u]<-1ex>_{\cal F}
	\ar@<-1.0ex>@{}[r]^(.13){}="a"^(.87){}="b" \ar_{\cal F} "a";"b"="13"
	&\mathrm{\{Lie\}}
	\ar@<-0.75ex>@{}[l]^(.13){}="a"^(.87){}="b" \ar_{U} "a";"b"="14"
	\ar@{}"13";"14"|(.2){\,}="15"
	\ar@{}"13";"14"|(.8){\,}="16"
	\ar@{}"15" ;"16"|{\top} 
	%%%%%%%%%%%%%%%%%%%
	% 	\ar@{}"1";"2"|(.2){\,}="3"
	% 	\ar@{}"1";"2"|(.8){\,}="4"
	% 	\ar@{}"3" ;"4"|{\top}
	%%%%%%%%%%%%%%%%%%%%
	\ar@/_7pc/@{<-}[]!<1ex,0ex>;[lllu]!<-2.5ex,0ex>^{ Prim }="1"
	\ar@/_8pc/[]!<2ex,0ex>;[lllu]!<-3.5ex,0ex>_{ \cal F^\sharp }="2"
	\ar@{}"1";"2"|(.2){\,}="3"
	\ar@{}"1";"2"|(.8){\,}="4"
	\ar@{}"3" ;"4"|{\bot}
	\ar@/^1pc/@{<-}[]!<4ex,0ex>;[lu]!<3ex,0ex>^(.75){U_2}="1"   
	\ar@/_0.9pc/@{->}[]!<0ex,0ex>;[lu]!<3ex,0ex>_(.7){\iota_2}="2"
	\ar@{}"1";"2"|(.17){\,}="3"
	\ar@{}"1";"2"|(.83){\,}="4"
	\ar@{}"3" ;"4"|{\scalebox{0.8}{\rotatebox{138}{$\top$}}}	
}
%%%%%%%%%%%%%%%%%%%%
\end{equation}
\rev{where $\cal F$ is the free functor, $U$ is the forgetful functor, $\iota$ is the inclusion, and the forgetful functors $U_1\colon  \mathrm{\{Grp\ltimes Lie\} \to \{Grp\}}$ and $U_2\colon \mathrm{\{Grp\ltimes Lie\} \to \{Lie\}}$ selecting the first and second component of $(G,L,\tr)$. In addition, $\iota_1$ and $\iota_2$ being $\iota_1(G)=G \mapsto (G,\{0\},\tr)$ and $\iota_2(L)= (\{1\},L,\tr)$. These two actions are the trivial ones.}

\medskip

Crossed modules of Hopf algebras are introduced in \citep{SM2} as encoding strict 2-quantum groups. \rev{Independently, there exists another definition of crossed modules from the perspective of symmetric monoidal categories given in \citep{VN1} which is different from the definition of Majid, but coincides with it in the case of cocommutative Hopf algebras.} \rev{Furthermore,} a more general notion is given in \citep{zbMATH05942003} and there is no agreement as to the unique crossed module definition for Hopf algebras; see \citep{zbMATH06932235} for the discussion. \rev{In this paper}, we follow \citep{SM2} \rev{because of} the following additional facts: it is clear from \citep{zbMATH07075977} that Majid's definition is \rev{not only coherent with \citep{VN1}, but also from the point of view of semi-abelian categories}. Moreover, it is proven in \citep{JFM1} that \edit{this} crossed module structure is also preserved under the functors $Gl$ and $Prim$. Therefore, \edit{following Majid's definition}, crossed modules of Hopf algebras can be seen as a unification of crossed modules of groups and of Lie algebras.

\medskip

The major outcome of this paper is to define 2-crossed modules of cocommutative Hopf algebras, which extend crossed modules to one level further. From a categorical point of view, this notion will unify the theory of 2-crossed modules of groups and of Lie algebras when we take the functors $Gl$ and $Prim$ into consideration. As for the group and Lie algebra case, we find out the functorial relationship between simplicial objects and 2-crossed modules in the category of cocommutative Hopf algebras. For this aim, we first give the explicit definition of a Moore complex of a simplicial cocommutative Hopf algebra, which will be constructed via Hopf kernels. No doubt this definition again unifies the Moore complex of groups and Lie algebras in the sense of the same functors. Then we obtain a 2-crossed module structure from a simplicial cocommutative Hopf algebra with Moore complex of length 2 with the aid of iterated Peiffer pairings. Consequently, we obtain the functor $\mathrm{\{SimpHopf_{\leq 2}^{cc}\} \to \{X_2Hopf^{cc}\}}$. On the other hand, we already have the functor $\mathrm{\{SimpGrp_{\leq 2}\} \to \{X_2Grp\}}$ from the category of simplicial groups with Moore complex of length two, to \rev{the category of} 2-crossed modules of groups \citep{MP2}; and similarly \rev{we have} $\mathrm{\{SimpLie_{\leq 2}\} \to \{X_2Lie\}}$ for the category of Lie algebras \citep{E1}. Finally, these two functors meet in the same diagram that proves the coherence of our 2-crossed module definition as being:
\begin{align*}
\xymatrix@R=25pt@C=35pt{
	\mathrm{\{SimpGrp_{\leq 2}\}} \ar[d] & \mathrm{\{SimpHopf_{\leq 2}^{cc}\}} \ar[r] \ar[d] \ar[l] & \mathrm{\{SimpLie_{\leq 2}\}} \ar[d] \\
	\mathrm{\{X_{2}Grp\}} & \mathrm{\{X_{2}Hopf^{cc}\}} \ar[l] \ar[r] & \mathrm{\{X_{2}Lie\}}  
}
\end{align*}
in which the horizontal arrows are extended from $Gl$ and $Prim$, respectively. \rev{Recall that, the outer vertical arrows in this diagram are known to be equivalences of categories (essentially by construction). We add to this not only the existence of a middle vertical arrows that makes everything commute, but also the result that the middle arrow is an equivalence as well.} 

%\edit{We also obtain a functor for the other way around, i.e. $\mathrm{X_2Hopf \to SimpHopf_{\leq 2}}$ which yields the equivalence of these two categories, as in the group and Lie algebra case. Of course, we primarily give the crossed module case of this functorial relationship before dealing with 2-crossed modules.}

\medskip

This study is the first step towards enhancing our understanding of Hopf algebras in terms of category theory and algebraic topology for higher dimensions, where the crossed modules can be considered as one dimensional categorical objects. As it stands, there already exist many higher dimensional categorical objects defined in the categories of groups and of Lie algebras. We are confident that the results of the present paper will serve as a base from which to unify these other higher dimensional structures and their properties in the category of (cocommutative) Hopf algebras. 

\section*{Acknowledgement}

The author expresses his sincere gratitude to Jo\~ao Faria Martins for kindly suggesting the problem and for his invaluable mentorship throughout the study. The author is also thankful to John Bourke and the anonymous referee for their obliging comments on the paper. 
The hospitality shown by Roger Picken (Instituto Superior T\'ecnico) and Centro de Matem\'atica e Aplica\c{c}\~oes (Universidade Nova de Lisboa), where indeed the fundamentals of the present work had been accomplished during 2014-2017, is deeply appreciated.
The author was supported by the projects Group Techniques and Quantum Information (MUNI/G/1211/2017) and MUNI/A/1160/2020 by Masaryk University.

\section{Quick Review of Hopf Algebras}

All Hopf algebras will be defined over \rev{an arbitrary} field $\kappa$. Towards the end of the paper, we will mainly work in the cocommutative setting. 

\subsection{Hopf algebraic conventions} 

\revt{Roughly speaking, a ``Hopf algebra" $H$ is a bialgebra with an antipode \citep{SM1}.} In full, \revt{it is a sextuple} $H=(H,\mu,\eta,\delta,\epsilon,S)$ where $H$ is a $\kappa$-vector space \revt{together with the following data:}
\begin{itemize}[leftmargin=.4cm]
	\item $(H,\mu,\eta)$ is a unital associative algebra. Thus
	\begin{itemize}
		\item $\mu\colon H \tn H \to H$ is an associative product. In short the product in $H$ induces a map $\mu\colon H\tn H \to H$, where $x\tn y\mapsto xy$.  
		\item $\eta\colon \kappa \to H$ is an algebra map endowing $H$ with a unit.  In short \mbox{$\eta \colon \lambda \in  \C \mapsto \lambda 1_H \in H$} (Here $1_H$ is the identity element of $H$).
	\end{itemize}
	\item $(H,\delta,\e)$ is a counital coassociative coalgebra. Thus 
	\begin{itemize}
		\item $\delta \colon H \to H \tn H$ is a coassociative coproduct. We use Sweedler's \rev{sigma} notation \rev{as in \citep{CK1} to denote the coproduct. \revt{Explicitly,} if $x \in H$, then the element $\delta (x) \in H \tn H$ will be written in the following form\footnote{\revt{In fact, the comultiplication of a coalgebra is denoted by $\Delta(x)$ in \cite{CK1}. However, to avoid the potential confusion with the simplicial $\Delta$, we rather use the notation $\delta(x)$ throughout the text.}}:}
		\begin{align*} 
		\delta (x) = \underset{(x)}{\sum } \,\, x' \otimes x''.
		\end{align*}
%		\begin{align*} 
%		\delta (x) = \underset{(x)}{\sum } \,\, x' \otimes x'', \textrm{ where } x \in H.
%		\end{align*}
		\item $\epsilon \colon H \to \kappa$ is the counit. So, for all $x \in H$, we have $$\displaystyle x=\sum_{(x)} \epsilon(x')x''=\sum_{(x)}
		x'\epsilon(x'').$$
	\end{itemize}
	\item  \revt{Additionally, the following hold:} 
	\begin{itemize}
		\item $\eta$ and $\mu$ are coalgebra morphisms,
		\item $\epsilon$ and $\delta$ are algebra morphisms. 
	\end{itemize}
	 \rev{In fact, these two statements are equivalent;} 	\revt{and we call the quintuple $H=(H,\mu,\eta,\delta,\epsilon)$ a bialgebra.}
	%(In order to prove $a)$ and $b)$ it suffices proving that $a)$ or $b)$ holds.) 
	%\item If $(H,\mu,\eta,\Delta,\epsilon)$ is a bialgebra then the space of linear maps $\hom_\kappa(H,H)$ is itself a unital associative algebra under the convolution product $(f*g)(x)=\sum_{(x)}f(x')g(x'')$. The unit of the convolution algebra is $\eta\circ \epsilon$. And 
	\item There exists an antipode, \rev{namely} an (inverse-like) anti-homomorphism $S\colon H \to H$ at the level of algebra and coalgebra, satisfying
	$$\sum_{(x)} S(x')x''=\sum_{(x)} x'S(x'')=\epsilon(x)1_H .$$
\end{itemize}
Moreover:
\begin{itemize}
	\item A Hopf algebra $H$ is said to be ``cocommutative" if, for all $x \in H$, we have
	\begin{align*}
	\underset{(x)}{\sum } \,\, x' \otimes x'' = \underset{(x)}{\sum } \,\, x'' \otimes x' .
	\end{align*}
	
	\item \rev{A Hopf algebra morphism (map) is exactly a bialgebra morphism, since bialgebra morphisms automatically preserve antipodes.} 
	
	\item Let $H$ be \rev{\textit{any}} Hopf algebra. An element $x\in H$ is said to be:
	\begin{itemize}
		\item primitive, if $\delta(x)=x \tn 1_H +1_H \tn x$,
		\item group-like, if $\delta(x)=x \tn x$ \revt{and $\e (x)=1_H.$}
	\end{itemize}
	\rev{It can be easily proven that $\e(x)=0$, if $x$ primitive. Then we have the set of primitive elements $Prim(H)$ that gives rise to a Lie algebra with the usual commutator bracket $[x,y]=xy-yx$. Similarly, the set of group-like elements $Gl(H)$ gives rise to a group with the product of $H$, where the inverse of any group-like element is $x^{-1}=S(x)$.}

	%	If $x$ is group-like then $\epsilon(x)=1$ and $S(x)=x^{-1}$; and if $x$ is primitive then $\epsilon(x)=0$ and $S(x)=-x$. The set of primitive elements $Prim(H)$ defines a Lie algebra, and the set of group-like elements $Gl(H)$ defines a group. 
	
	\item \rev{Conversely:}
	
	\begin{itemize}
		\item \rev{Let L be a Lie algebra. The universal enveloping algebra $U(L)$ turns into a Hopf algebra where $\delta(x)=x \tn 1 +1 \tn x$, $\e(x)=0$ and $S(x)=-x$, $\forall x \in L$.}
		\item \rev{Let $G$ be a group. We have the group algebra $\k[G]$ given by the free vector space on $G$ together with the multiplication induced by the group operation. In addition, $\k[G]$ forms a Hopf algebra structure together with $\delta (g) = g \otimes g$, $\e(g)=1$ and $S(g)=g^{-1}$, on the base elements $g \in G$.}
	\end{itemize}
	\rev{Moreover, (considered as Hopf algebras) both the group Hopf algebra and the universal enveloping algebra come equipped with the cocommutativity property.}
	
	\item Thus, we have the functors
	\begin{align}\label{main-adjunction}
	\xymatrix@C=30pt@R=30pt{     
		\mathrm{\{Grp\}}
		\ar@<-1.0ex>@{}[r]^(.26){}="a"^(.67){}="b" \ar_{\cal G^\sharp} "a";"b"="9" 
		&\mathrm{\{Hopf^{cc}\}}
		\ar@<-0.75ex>@{}[l]^(.33){}="a"^(.74){}="b" \ar_{Gl} "a";"b"="10"
		\ar@{}"9";"10"|(.2){\,}="11"
		\ar@{}"9";"10"|(.8){\,}="12"
		\ar@{}"12" ;"11"|{\top}  
		% 	\ar@{<-}[u]<1.3ex>^{U}_{\vdash}
		% 	\ar[u]<-1ex>_{\cal F}
		\ar@<-1.0ex>@{}[r]^(.34){}="a"^(.75){}="b" \ar_{Prim} "a";"b"="13"
		&\mathrm{\{Lie\}}
		\ar@<-0.75ex>@{}[l]^(.25){}="a"^(.66){}="b" \ar_{\cal F^\sharp} "a";"b"="14"
		\ar@{}"13";"14"|(.2){\,}="15"
		\ar@{}"13";"14"|(.8){\,}="16"
		\ar@{}"15" ;"16"|{\bot} 
	} 
	\end{align} 
	\rev{between the categories of groups, cocommutative Hopf algebras, and Lie algebras, respectively. For an explicit details on the constructions, see \citep{zbMATH06579662}[\S 2.2].}
	
	\item 	Let $H$ be a (cocommutative) Hopf algebra. A sub-Hopf algebra $A \subset H$ is a subvector space $A$, such that $\mu( A \tn A) \subset A$, $\delta(A) \subset A \tn A$ and $\eta(\kappa) \subset A$. Clearly a sub-Hopf algebra inherits a (cocommutative) Hopf algebra structure from $H$.
	
	\item A (cocommutative) sub-Hopf algebra $A$ of $H$ is called normal \citep{vespa}, if $x \ad a \in A$ for all $x \in H$ and $a \in A$. Here we put
	\begin{align*}
	x \ad a= \sum_{(x)} x' a S(x'' ) \, ,
	\end{align*}
	which is called the ``adjoint action".
\end{itemize}

\subsection{Hopf algebra modules and smash products}

\rev{Let $H$ be a Hopf algebra and $I$ be a vector space. $I$ is said to be an $H$-module with an action $\rh \colon H \tn I \to I$, explicitly $ \revt{ x \tn v  \mapsto x\rh v }$, satisfying:}
\begin{itemize}
	\item $1_H \rh v = v$, for all  $v \in I$, 
	\item $(xy) \rh v = x \rh (y \rh v)$, for all $x,y \in H$ and $v \in I$.
\end{itemize}
\rev{In addition to this, let $I$ be a bialgebra. Then $I$ is said to be an $H$-module algebra if:} 
\begin{itemize}
	\item $x \rh 1_I = \e(x)1_I$, for all $x \in H$, 
	\item $x \rh (uv) = \underset{(x)}{\sum } \,\, (x' \rh u) (x'' \rh v)$, for all $x \in H$ and $u,v \in I$,
\end{itemize}
\rev{and similarly, an $H$-module coalgebra if:}
\begin{itemize}
	\item\label{eq:einstein} $\delta (x \rh v) = \underset{(x)(v)}{\sum }  \,\, (x' \rh v') \tn (x'' \rh v'')$, for all $x \in H$, $v \in I$, 
	\item $\e (x \rh v) = \e (x) \e (v) $, for all $x \in H$, $v \in I$.
\end{itemize} 

\n The following is well known with a proof from \citep{SM1}.
\begin{definition}\label{tensor}
	Let $I,H$ be cocommutative Hopf algebras, where $I$ is an $H$-module algebra and coalgebra (one can call it $H$-module bialgebra) under the action $\rh \colon H \otimes I \to I$. Then we have a cocommutative Hopf algebra $I \otimes_{\rho} H$ called the ``smash product" with the underlying vector space $I \otimes H$ such that:
	\begin{itemize}
		\item $(u \otimes x) (v \otimes y)= \underset{(x)}{\sum } \,\, \big(u \, (x' \rh v) \big) \otimes x''y $,
		\item $\delta (u \otimes x)=\underset{(u)(x)} {\sum }  \,\, (u' \otimes x') \otimes  (u'' \otimes x'')$,
		\item $S(u \otimes x)= \big( 1_{I} \otimes S(x) \big) \, \big(S(u) \otimes 1_{H} \big)$,
	\end{itemize}
	with the identity $1_I \otimes 1_H$ and the co-identity $\e(u \tn x)=\e(u)\e(x)$.	
\end{definition}

\begin{example}\label{example-adj}
	If $H$ is a \textit{cocommutative}  Hopf algebra, then $H$ itself has a natural \mbox{$H$-module} algebra and coalgebra structure given by the adjoint action $$\rev{(x \tn y)} \in H \otimes H \mapsto x \ad y = \underset{(x)}{\sum } x'yS(x'') \in H.$$ 
\end{example}

\begin{proposition}
	If $H$ is cocommutative, the antipode $S \colon H \to H$ becomes an idempotent. We therefore obtain $S(x \ad y) = x \ad S(y)$. \rev{Furthermore, this property is still true for any action -- which is a direct consequence of the action conditions.}
\end{proposition}
	
\begin{remark}	
	Example \ref{example-adj} is not true in non-cocommutative case, since \rev{it does not form an $H$-module coalgebra}:
	\begin{align*}
	\delta (x \ad y)  = \underset{(x \ad y)}{\sum } \, (x \ad y)' \otimes (x \ad y)'' & = \underset{(x)(y)}{\sum } \, x'y'S(x'''') \otimes x''y''S(x''') \\
	& \neq \underset{(x)(y)}{\sum }  \,\, (x' \ad y') \tn (x'' \ad y'').
	\end{align*}
\end{remark}

\begin{example}
	\rev{For \textit{any} Hopf algebra $H$, the regular action $$\rev{(x \tn y)} \in H \otimes H \mapsto x \rh y = xy \in H$$ turns $H$ into an $H$-module coalgebra} \revt{but not an $H$-module algebra. In other words, this example means that the multiplication is a coalgebra morphism but not an algebra morphism. }
	
	\medskip 
	
	\revt{We also have a trivial action which turns $H$ into an $H$-module algebra and coalgebra which is given by $x \rh y = \epsilon (x)y$.}
\end{example}

\begin{example}\label{Cartier}
	\rev{Recall the functors given in \eqref{main-adjunction}. Let $L$ be a Lie algebra and $i\colon L \to U(L)$ be the inclusion. If there exists an action $\tr$ of a group $G$ on $L$ by Lie algebra maps \citep{zbMATH06985457}, then ${\cal F}^\sharp(L)$ becomes a ${\cal G}^\sharp(G)$-module bialgebra where the action is defined on the base elements by \revt{$g \rh v$}, for each $v \in L$ and $g \in G$. Consequently, we \revt{can form} ${\cal F}^\sharp(L)\otimes_{\rho} {\cal G}^\sharp(G)$. If the base field is algebraically closed with zero characteristic, then we have \revt{an} isomorphism  $H \cong {\cal F}^\sharp(Prim(H))\otimes_{\rho} {\cal G}^\sharp(Gl(H))$ given in \citep{Cartier}.}
\end{example}

%\begin{example}[Cartier-Gabriel structure theorem]\label{Cartier}
%	\rev{If $L$ is a Lie algebra the universal enveloping algebra, seen as a Hopf algebra, is denoted by ${\cal F}^\sharp(L)$. Let $i\colon L \to U(L)$ be the inclusion. If $G$ is a group, the group algebra of $G$, seen as a Hopf algebra is denoted by ${\cal G}^\sharp$. If $G$ acts on $L$ by Lie algebra maps then ${\cal F}^\sharp(L)$ is a ${\cal G}^\sharp$-module bialgebra, where if $g \in G$ and $L$ in $L$ the action $\rho$ takes the form $g \t i(v)=i\big(\rho_g(v)\big)$, where $v \in L$ and $g \in G$. The well-known Cartier-Gabriel structure theorem \citep{Cartier} states that if $\kappa$ is algebraically closed with zero characteristic all cocommutative Hopf algebras $H$ are of the following form, up to isomorphism: $${\cal F}^\sharp(L)\otimes_{\rho} {\cal G}^\sharp(G).$$ 
%	(And if $H$ is a cocommutative coalgebra then $L$ should be taken to be  the Lie algebra of primitive elements and of course $G$ is the group of group-like elements.) }
%\end{example}

\subsection{Hopf kernels}

Some categorical properties of the cocommutative Hopf algebras we need in this paper are given below. %\rev{At this point, we would like to emphasize that, everything in this subsection is only valid for cocommutative Hopf algebras}. 
For more details we refer to \citep{AD1,AgoreII,vespa,HP1}.

\begin{itemize}	
	\item \rev{The category $\mathrm{\{Hopf^{cc}\}}$ of cocommutative Hopf algebras over an arbitrary field is semi-abelian. Moreover, it is a full subcategory being complete and cocomplete. For more details, see \citep{zbMATH07075977}.}
	
	\item \rev{The zero object is $\kappa$, considered as a cocommutative Hopf algebra with the obvious structure maps.} Given Hopf algebras $A$ and $B$, the zero map $z_{A,B}\colon A \to B$ is $\eta_B \epsilon_A$.   
	
	\item \edit{The categorical product of cocommutative Hopf algebras $A$ and $B$ is $A \tn B$ with underlying vector space \rev{$A\otimes B = A \otimes_{\k} B$}. We have two projections $a \tn b \in A \tn B \mapsto a \epsilon(b) \in A$ and $a \tn b \in A \tn B \mapsto \epsilon(a) b \in B$. The extension to a finite number of components is the obvious one.}
	
	\item \rev{The equalizer (in the category of \textit{all} Hopf algebras) of $f,g \colon A \to B$ is given by
	\begin{align*}
		\{x \in A\colon \sum_{(x)} x'\tn f(x'') \tn x''' = \sum_{(x)} x' \tn g(x'') \tn x''' \} \, ,
	\end{align*}
	which turns out to be, in the \textit{cocommutative} case:
	\begin{align*}
	\{x \in A\colon \sum_{(x)} x'\tn f(x'')=\sum_{(x)} x' \tn g(x'')\}=\{x \in A\colon \sum_{(x)} f(x')\tn x''=\sum_{(x)} g(x') \tn x'' \} \, .
	\end{align*} }
	
	\item \rev{From this description of the equalizer, we can obtain the kernel (in the category of \textit{all} Hopf algebras) as the equalizer of $f$ and $z_{A,B}$ as being
	\begin{align*}
		 \{x \in A\colon \sum_{(x)} x'\tn f(x'') \tn x''' = \sum_{(x)} x' \tn 1_B \tn x'' \} \, .
	\end{align*}
	However, in the \textit{cocommutative} case, this can be simplified to:\footnote{In general case, these sets only define subalgebras of $H$ which are denoted by $\mathrm{RKer}(f)$ and $\mathrm{LKer}(f)$, respectively \citep{AD1}. Regarding the cocommutative setting: since $\mathrm{HKer}(f)$ is the kernel in the category of cocommutative Hopf algebras, one can also denote it by $\ker (f)$. To avoid the confusion with the usual linear kernel, we prefer using the notation $\lk(f)$.}
	\begin{align}\label{hopfkernel}
		\mathrm{HKer}(f) = \{x \in A\colon \sum_{(x)} x'\tn f(x'') =x \tn 1_B \} = \{x \in A\colon \sum_{(x)} f(x') \tn x'' = 1_B \tn x \} \, .
	\end{align} }
	
	%\edit{Explicitly, the category of cocommutative Hopf algebras is complete (cocomplete).}
	
%	\item The (Hopf) kernel of a map $f\colon A \to B$ is: 
%	\begin{align}\label{hopfkernel}
%	eq(f,z_{A,B}) = \{x \in A\colon \sum_{x} x'\tn f(x'') =x \tn 1 \},
%	\end{align}
%	which will be called $\lk (f)$ for short. Remark that: $$\lk (f) \doteq eq(f,z_{A,B}) \cong eq(z_{A,B},f) \doteq \rk (f) ,$$ because of the cocommutativity. Hence, two different possible Hopf kernel notions coincide in the category of cocommutative Hopf algebras. However, none of them defines the Hopf kernel for non-cocommutative case.
	
	\item Consider the Hopf kernel  $\lk (f\colon A \to B)$. If we apply $\mu (\e \otimes \id)$ in \eqref{hopfkernel}, we obtain
	\begin{align*}
	x \in \lk (f) \implies f(x) = \e(x) 1_B \, ,
	\end{align*}
	that means Hopf kernels are specific cases of linear kernels.
	
	\item The functors $Gl$ and $Prim$ preserve kernels. In other words, we have
	\begin{align}\label{preservekernels}
		\rev{Gl(\mathrm{HKer}(f)) = \ker (Gl(f)) \quad \text{and} \quad Prim(\mathrm{HKer}(f)) = \ker (Prim(f)) \, ,}
	\end{align}
%	\begin{align}\label{preservekernels}
%			\xymatrix@R=40pt@C=30pt{ \lk (f) \ar@<0.5ex>[r]^{Gl} \ar@<-0.5ex>[r]_{Prim} & \ker (f)\, ,
%			}
%	\end{align}
	for the corresponding categories.
	
	\item For any Hopf algebra map $f \colon A \to B$, Hopf kernel $\lk(f)$ defines a normal sub-Hopf algebra. 
	
	\item $\lk(\id)=\k$. \rev{More generally, if a Hopf algebra map $f$ is injective (therefore monic), then we have $\lk(f) = \k$.}
	
\end{itemize}

\subsection{Crossed modules of cocommutative Hopf algebras}

The following definition is well established as a crossed module of cocommutative Hopf algebras \rev{that independently introduced in \citep{SM2} and \citep{VN1} which are coherent in the cocommutative setting}.

\begin{definition} \label{xmod}
	A ``crossed module" of cocommutative Hopf algebras is given by a Hopf algebra map $\d \colon I \to H$ \rev{together with an action of $H$ of $I$ denoted by $\, \tr_{\rho} \colon H \tn I \to I$ which turns $I$ into an $H$-module bialgebra}, satisfying:
	\begin{itemize}
		\item[-] $\d(x \rh v) = x \ad \d(v)$,
		\item[-] $\d(u) \rh v = u \ad v$,
	\end{itemize}
	for all $x \in H$ and $u,v \in I$.\footnote{In general, there is an extra crossed module condition called ``compatibility" that automatically holds in the cocommutative setting.} \edit{Without the second condition, we call it a precrossed module.}
\end{definition}

\begin{remark}
	Group-like and primitive elements preserve crossed module structures. Consequently, we can write down the crossed module conditions of groups and Lie algebras in the sense of the functors $Gl$ and $Prim$ as follows:
	\begin{itemize}
		\item A crossed module of groups is given by a group homomorphism $\d\colon E \to G$ together with an action $\tr$ of $G$ on $E$ \rev{by automorphisms}, such that:
		\begin{align*}
		\d(g \tr e) = g \,\d(e) \, g^{-1}, \quad
		\d(e)  \tr f= e\, f \, e^{-1},
		\end{align*}
		for all $e,f \in E$ and $g \in G$.
		\item A crossed module of Lie algebras (i.e. differential crossed module) is given by a Lie algebra homomorphism $\d\colon \mathfrak{e} \to \mathfrak{g}$ together with an action $\tr$ of $\mathfrak{g}$ on $\mathfrak{e}$ \rev{by derivations}, such that:
		\begin{align*}
		\d(g \tr e)  = [g ,\d(e)] , \quad
		\d(e)  \tr f  = [e, f],
		\end{align*}
		for all $e,f \in \mathfrak{e}$ and $g \in \mathfrak{g}$.
	\end{itemize}
	Therefore, we have the functors:
	\begin{align}\label{functors}
	\xymatrix@R=40pt@C=30pt{ \mathrm{\{XLie\}}  &
		\mathrm{\{XHopf^{cc}\}} \ar[r]^{Gl} \ar[l]_{Prim} &
		\mathrm{\{XGrp\}} 
	}
	\end{align}
	between the categories of crossed modules of Lie algebras, cocommutative Hopf algebras, and groups, respectively. See \citep{JFM1} for more details.
\end{remark}

\section{The Moore Complex}

From now on, all Hopf algebras will be considered cocommutative and we just use the term ``Hopf algebra" for the sake of simplicity. 
%And in this subsection only $\Delta$ does not denote the coproduct map in a Hopf algebra.

\begin{definition}
	A chain complex \revt{$(A_{\bullet},\d_{\bullet})$} of Hopf algebras is given by a sequence of Hopf algebra maps 
	$$\to A_{n+1}\ra{\d_{n+1} } A_{n} \ra{\d_n} A_{n-1} \to \dots\to  A_1 \ra{\d_1} A_0,$$
	such that: $\forall n\geq 1$, we have \rev{$\d_{n}  \d_{n+1}=\eta_{A_{n-1}} \epsilon_{A_{n+1}}$}, and the latter is the zero morphism $x \in A_{n+1} \mapsto \e(x)1_{A_{n-1}} \in  A_{n-1}$. 
	Given a chain complex of Hopf algebras, then $\d_n(A_{n})$ is, a priori, only a  sub-Hopf algebra of $A_{n-1}$. We say that a chain complex of Hopf algebras is normal if $\d_n(A_{n})$ is a normal sub-Hopf algebra of $A_{n-1}$, for each positive integer $n$.
\end{definition}	

\subsection{Simplicial Hopf algebras}

\begin{remark}
	\edit{
	Recall that the category $\Delta$ is the category whose objects are the non-negative integers $n \in \mathbb{N}$ and whose morphisms $n \to m$ are the order preserving maps from $\{0,\dots,n\}$ to $\{0, \dots,m\}$. This category is generated by certain maps $D_i^n\colon \{0,\dots, n-1\} \to \{0,\dots, n\}$ and  $\rev{S_i^{n+1}} \colon \{0,\dots, n+1\} \to \{0,\dots, n\}$, where \rev{$n >0$ and $0 \leq i < n$}. In short $D_i^n\colon \{0,\dots, n-1\} \to \{0,\dots, n\}$ is injective and its image does not include $i$, $S_i^{n+1}\colon \{0,\dots, n+1\} \to \{0,\dots,n\}$ is surjective, and $i$ is the only element with a double pre-image. These morphisms satisfy the well known (co)simplicial relations, \rev{which we will write, dually, below in \ref{simp-defn}}. A simplicial \rev{object}
	\citep{PM1}  in a category $\cal C$ is a covariant functor $F\colon \Delta^{\rm op} \to {\cal C}$.  And normally we put $d_i^n=F(D_i^n)$ and $s_i^{n+1}=F(S_i^{n+1})$.}

	\medskip
	
	\edit{
	We have a category ${\cal C}^{\Delta^{op}}$, whose objects are the functors $\Delta^{op} \to {\cal C}$ (i.e. the simplicial objects in $\cal C$) and whose morphisms are the natural transformations between functors. If ${\cal C}$ is complete and cocomplete \rev{(as will always be the case in this paper)}, then   ${\cal C}^{\Delta^{op}}$ is complete and cocomplete and all limits and colimits can be computed pointwise.} 
\end{remark}

\begin{definition}\label{simp-defn}
	A simplicial Hopf algebra $\mathcal{H}$ is therefore a simplicial \rev{object} in the category of Hopf algebras. In other words, it is given by a collection of Hopf algebras $H_{n}$ $(n\in \mathbb{N})$ together with Hopf algebra maps called faces and degeneracies, \rev{respectively}
	\begin{align*}
	\begin{tabular}{llll}
	$d_{i}^{n} \colon$ & $H_{n}\to H_{n-1}$ & $,$ & $0\leq i\leq n $ \\
	$s_{j}^{n+1} \colon $ & $H_{n}\to H_{n+1}$ & $,$ & $0\leq j\leq n$%
	\end{tabular}
	\end{align*}
	which are to satisfy the following simplicial identities:\,\footnote{\rev{To avoid overloaded notation, we will not use superscripts for faces and degeneracies.}}
	\begin{equation*}
	\begin{tabular}{llll}
	(i) & $d_{i}d_{j}=d_{j-1}d_{i}$ & if & $i<j$ \\
	(ii) & $s_{i}s_{j}=s_{j+1}s_{i}$ & if & $i\leq j$ \\
	(iii) & $d_{i}s_{j}=s_{j-1}d_{i}$ & if & $i<j$ \\
	& $d_{j}s_{j}=d_{j+1}s_{j}= \id$ &  &  \\
	& $d_{i}s_{j}=s_{j}d_{i-1}$ & if & $i>j+1$ \\
	\end{tabular}%
	\end{equation*}
	
	%	Any simplicial Hopf algebra can be pictured as:
	%	\begin{align*}
	%	\xymatrix@C=40pt{
	%		\mathcal{H} \: = \: \: \ar@{.}[r] &
	%		{H}_{3}
	%		\ar@<2.25ex>[r] \ar@{.>}@<1.5ex>[r] \ar@{.>}@<0.75ex>[r] \ar@<0ex>[r] &
	%		{H}_{2}\ar@<3ex>[r]|{d_2} \ar@<1.5ex>[r]|{d_1} \ar@<0ex>[r]|{d_0}
	%		\ar@/^1pc/[l] \ar@{.>}@/^1.5pc/[l] \ar@/^2pc/[l]&
	%		{H}_{1}\ar@<1.5ex>[r]|{d_1} \ar[r]|{d_0}
	%		\ar@/^1pc/[l]|{s_0} \ar@/^1.5pc/[l]|{s_1} &
	%		{H}_{0}\ar@/^1pc/[l]|{s_0} }
	%	\end{align*}
	
	A simplicial Hopf algebra can be pictured as:
	\begin{align}\label{simphopf}
	\xymatrix@C=40pt{
		\mathcal{H} \: = \: \: \ar@{.}[r] &
		{H}_{3}
		\ar@<3ex>[r] \ar@<2ex>[r] \ar@<1ex>[r] \ar@<0ex>[r] &
		{H}_{2}\ar@<3ex>[r]|{d_2} \ar@<1.5ex>[r]|{d_1} \ar@<0ex>[r]|{d_0}
		\ar@/^1pc/[l] \ar@/^1.5pc/[l] \ar@/^2pc/[l]&
		{H}_{1}\ar@<1.5ex>[r]|{d_1} \ar[r]|{d_0}
		\ar@/^1pc/[l]|{s_0} \ar@/^1.5pc/[l]|{s_1} &
		{H}_{0}\ar@/^1pc/[l]|{s_0} } 
	\end{align}
	and we denote the category of simplicial Hopf algebras by $\mathrm{\{SimpHopf^{cc}\}}$. 
	%	An $n$-truncated simplicial cocommutative Hopf algebra $$Tr_{n}\mathcal{H}=\{H_{n},\ldots,H_{1},H_{0}\}$$ is defined like a simplicial cocommutative Hopf algebra except that we forget what happens in degree greater than $n$. 
	%We denote the corresponding category by $\mathbf{ntrSimpHopf}$.
\end{definition}

\begin{remark}
	\edit{ An $n$-truncated simplicial Hopf algebra is defined like a simplicial Hopf algebra except that we forget what happens in degree greater than $n$. More precisely it is  a contravariant functor from $\Delta^n$ (the full category of $\Delta$ whose objects are the naturals $\leq n$) into the category of Hopf algebras. We denote the category of $n$-truncated simplicial Hopf algebras by $\mathrm{\{SimpHopf_{\mid n}^{cc}\}}$.}
	
	\medskip
	\edit{
	The forgetful functor  $\mathrm{\{SimpHopf^{cc}\}} \to \mathrm{\{SimpHopf_{\mid n}^{cc}\}}$ has a left adjoint \rev{given by left Kan extension, and a right adjoint given by right Kan extension. These are called the skeleton and coskeleton functors, respectively, and are contructed using colimits and limits (as first defined in \citep{Duskin})}. Let us quickly revise the construction of the coskeleton functor for Hopf algebraic case: For a given $n$-truncated simplicial Hopf algebra ${\mathcal{H}^{\mid n}}$, consider the Hopf algebra
	\rev{$H_{n}^{n+1}$}. Then obtain for each pair $0\leq i < j\leq n+1$ the equalizer $M_{i,j}$ of $d_i \, p_j$ and  $d_{j-1} \, p_i$ \rev{where $p_i,p_j$ denote the projections}. And then consider the intersection  $H_{n+1}$ of all of the \rev{$M_{i,j}$}. \rev{Naturally, the face maps $H_{n+1} \to H_{n}$ are given by the projections $p_i$ where ($0\leq i \leq n+1$) and degeneracies $H_{n} \to H_{n+1}$ are uniquely defined by a generator $\alpha_{i,j}$ such that $p_i s_j = \alpha_{i,j}$, for all $0\leq j \leq n$.} This yields an \mbox{$(n+1)$-truncated} simplicial Hopf algebra \rev{that is so-called the simplicial kernel}. And proceding inductively, we thus define the coskeleton functor. }
\end{remark}

%\subsection{The Moore complex}

\begin{lem}\label{chain-defn}
	Given a simplicial Hopf algebra \eqref{simphopf}, we have the chain complex of Hopf algebras \revt{$(NH_{\bullet},\d_{\bullet})$} given by:
	\begin{itemize}
		\item $NH_0 = H_0$,
		\item $NH_{n}=\underset{i=0}{\overset{n-1}{\mathlarger{\mathlarger{\mathlarger{{\cap}}}}}} \lk(d_{i})$, for $n \geq 1$,
		\item $\partial _{n} \colon NH_{n}\to NH_{n-1}$ as being the restriction of $d_{n}$ to $NH_n$.
	\end{itemize}
\end{lem}

\begin{proof}
	\rev{The proof is evident. We refer to the respective results given in \citep{zbMATH02125642} for semi-abelian categories.}
\end{proof}

%\begin{proof}
%	We know that the intersection of any number of sub-Hopf algebras is again a sub-Hopf algebra. Thus $NH_i$ defines a sub-Hopf algebra due to \eqref{hopfkernel}. Moreover $\d_n$ is well defined, since for all $x \in NH_n$ and $i<n$, we have
%	\begin{align*}
%	(d_{i}\otimes id)\delta (d_{n}(x)) & = (d_{i}\otimes id)\underset{(x)}%
%	{\sum }d_{n}(x')\otimes d_{n}(x'') \\ 
%	& = \underset{(x)}{\sum } ( d_{i} d_{n} ) (x')\otimes d_{n}(x'') \\ 
%	& = \underset{(x)}{\sum } ( d_{n-1} d_{i} ) (x')\otimes d_{n}(x'') \\ 
%	& = (d_{n-1}\otimes d_{n})\underset{(x)}{\sum }d_{i}(x')\otimes
%	x'' \\ 
%	& = (d_{n-1}\otimes d_{n})(1\otimes x) \\ 
%	& = 1\otimes d_{n}(x).
%	\end{align*}
%	that yields $\im (\d_n) \in \lk(d_i)$ for all $i<n$; therefore $\d_n (x) \in NH_{n-1}$. 
%	
%	\medskip
%	
%	On the other hand, to have a chain complex, we need to prove $$( \d_n  \d_{n+1} ) (x)=\epsilon(x)1_{N_{n-1}} \, ,$$ for all $x \in NH_{n+1}$. For this aim, we get
%	\begin{align*}
%	(\d_n \tn id) \, \delta \left( \d_{n+1} (x) \right) & = (\d_n \tn id) \, \underset{(x)}{\sum } \d_{n+1} (x') \tn \d_{n+1} (x'') \\
%	& = \underset{(x)}{\sum } ( \d_n  \d_{n+1} ) (x') \tn \d_{n+1} (x'') \\
%	&= \underset{(x)}{\sum } ( \d_n  \d_{n} ) (x') \tn \d_{n+1} (x'') \\
%	& =  (\d_n \tn \d_{n+1}) \, \underset{(x)}{\sum } \d_{n} (x') \tn x'' \\
%	& = (\d_n \tn \d_{n+1}) \, (1 \tn x) \\
%	& = 1 \tn \d_{n+1} (x) .
%	\end{align*}
%	Thus $\im (\d_{n+1}) \in \lk(\d_n)$ that implies $( \d_n \d_{n+1} ) (x)=\epsilon(x)1_{N_{n-1}}$ from \eqref{hopfkernelprop}.  
%\end{proof}

\begin{definition}[Moore Complex]
	For a given simplicial Hopf algebra $\H$, the chain complex \revt{$(NH_{\bullet},\d_{\bullet})$} will be called the ``Moore complex" of $\H$.
\end{definition}

The Moore complex of groups is also known as a normalized chain complex of simplicial groups in the literature \citep{MP2}. The following lemma proves that the Moore complex of a simplicial Hopf algebra has the same property.

\begin{lem}
	The Moore complex \revt{$(NH_{\bullet},\d_{\bullet})$}  of a simplicial Hopf algebra $\H$ is a normal chain complex. 
\end{lem}

\begin{proof}
	\rev{As in the proof of Lemma \eqref{chain-defn}, we refer to \citep{zbMATH02125642}.}
\end{proof}

%\begin{proof}
%	For any $\d_{n+1} \colon NH_{n+1} \to NH_n$, we have to prove that $\im (\partial_{n+1}) $ is a normal sub-Hopf algebra of $NH_n$.
%	
%	\medskip
%	
%	It is already clear that $\im(\partial_{n+1}) $ is a sub-Hopf algebra. So we only need to show that it is normal, i.e. $x \ad b \in NH_n$, for all $x \in NH_n$ and $b \doteq \d_{n+1}(a) \in im(\partial_{n+1})$. For this aim, we get
%	\begin{align*}
%	(d_n \tn id ) \, \delta (x \ad b) & =(d_n \tn \id ) \, \underset{(x)}{\sum } \, x'b'S(x'') \tn x''' b'' S(x'''') \\
%	& \doteq (d_n \tn id ) \, \underset{(x)}{\sum } \, x' \d_{n+1} (a') S(x'') \tn x''' \d_{n+1} (a'') S(x'''') \\
%	& = \underset{(x)}{\sum } \, d_n(x') \, (d_{n} ( d_{n+1} (a') ) \, d_n (S(x'')) \tn x''' \, d_{n+1} (a'') \, S(x'''') \\
%	& = \underset{(x)}{\sum } \, d_n(x') \, d_{n} ( d_{n}  (a') ) \, d_n (S(x'')) \tn x''' \, d_{n+1} (a'') \, S(x'''') \\
%	& = \underset{(x)}{\sum } \, d_n(x')  \, d_n (S(x'')) \tn x''' \, d_{n+1} (a) \, S(x'''') \\
%	& = 1 \tn (x \ad b) 
%	\end{align*}
%	Note that we used the Hopf kernel property of $a \in NH_{n+1}$ for cancellation. 
%\end{proof}

\begin{definition} We say that a simplicial Hopf algebra has Moore complex of length $n$, if $N H_i$ is zero object, for all $i>n$. 
	
	\medskip
	
	We denote the corresponding category by $\mathrm{\{SimpHopf_{\leq n}^{cc}\}}$.
\end{definition}

\begin{proposition}\label{final1}
	Following the property \eqref{preservekernels}, one can say that the functors $Gl$ and $Prim$ preserve the Moore complex definition, as well as the length of it. We therefore have the functors
		\begin{align*}
		\xymatrix@R=40pt@C=30pt{ \mathrm{\{SimpLie_{\leq n}\}}  &
			\mathrm{\{SimpHopf_{\leq n}^{cc}\}} \ar[r]^{Gl} \ar[l]_{Prim} &
			\mathrm{\{SimpGrp_{\leq n}\}} 
		}
		\end{align*}
		with referring to \citep{C2,E1} for the corresponding categories.
\end{proposition}

%\subsection{The simplicial decomposition}\label{decompose}

\begin{remark}
	Let $\partial \colon I \to H$ and $i \colon H \to I$ be Hopf algebra maps such that $\d \, i = \id_H$. \revt{Such a couple of maps is usually called a ``point" and we denote it by $(\d\colon I \to H,i)$.}
	
	\medskip
	
	From the categorical point of view, it is equal to say that $\d$ is a split epimorphism with a chosen splitting $i$. Together with the (Hopf) kernel of $\d$, it is a split extension:
	\begin{align*}
		\rev{ \xymatrix{ \k \ar[r] & \lk(\d) \ar@{^{(}->}[r] & I \ar[r]^{\d} & H \ar@/^/[l]^{i} \ar[r] & \k} }
	\end{align*}
\end{remark}

The following theorem is due to \citep{SM3, Rad1}.

\begin{thm}[Majid/Radford]\label{radford}
	Let $(\d\colon I \to H,i)$ be a \revt{point}. Then $I$ is an $H$-module bialgebra where the action $\rh \colon H \tn I \to I$ is the adjoint action via $i$.
	
	\medskip
	
	Moreover, we have an isomorphism of Hopf algebras $I \cong \lk(\d) \tn_{\rho} H$, where the maps below are mutually inverse:
	$$\Psi\colon v \in I \mapsto \sum_{(v)} f(v') \tn \d(v'') \in \lk(\d) \tn_{\rho} H ,  $$
	with $ f (x)= \underset{(x)}{\sum } \, x' \, i\d(S(x''))$, for all $x \in I$; and
	$$\Phi\colon a \tn x \in \lk(\d) \tn_{\rho} H \mapsto  a \, i(x) \in I.$$
\end{thm}

\begin{remark}
	By using simplicial identities, we can obtain \revt{points} $(d_i \colon H_{n+1} \to H_{n}, s_i)$ \revt{for each $n$, in a simplicial Hopf algebra \eqref{simphopf}}. Therefore, one can adapt Theorem \eqref{radford} to \rev{a} simplicial Hopf algebra as follows.
\end{remark}

\begin{proposition}\label{preiso}
	In a simplicial Hopf algebra $\H$, there exists an action
	\begin{equation*}
	\begin{tabular}{cll}
	$\rev{{\rh}_i} \colon s_i (H_{n-1}) \otimes \lk \big(d_i \big)$ & $\longrightarrow $ & $\lk \big(d_i \big)$ \\ 
	$(a,x)$ & $\longmapsto $ & $a \tsma{i} x = a \ad x \, ,$
	\end{tabular}
	\end{equation*}
	for all \revt{$0 \leq i \leq n-1$ and $0 < n$}. Consequently, we have the smash product Hopf algebra
	\begin{align*}
	\lk (d_{i}) \otimes_{\rho_{i}} s_{i}(H_{n-1}) \, ,
	\end{align*}
	from Definition \ref{tensor}. 
\end{proposition}

\begin{thm}\label{isomorphism} 
	For any simplicial Hopf algebra $\mathcal{H}$, we have an isomorphism
	\begin{equation*}
	H_{n}\cong \lk (d_{i}) \osma{i} s_{i}(H_{n-1}) \, ,
	\end{equation*}
	of Hopf algebras, for all $n \in \mathbb{N}$, $i \leq n-1$.
\end{thm}

\begin{proof}
	The map
	\begin{equation*}
	\begin{tabular}{rcll}
	$\phi :$ & $H_{n}$ & $\longrightarrow $ & $\lk (d_{i})\osma{i}
	s_{i}(H_{n-1})$ \\ 
	& $x$ & $\longmapsto $ & $\underset{(x)}{\sum } \,\, f_i (x') \otimes s_{i}d_{i}(x'')$%
	\end{tabular}
	\end{equation*}
	gives the isomorphism, where \rev{the} ``Hopf kernel generator map" $f_i\colon H_n \to \lk(\rev{d_i})$ is a linear map defined by $f_i \colon x \mapsto \sum_{(x)} x' s_{i}d_{i}(S(x''))$.
\end{proof}

\begin{deco}
	Consider the simplicial Hopf algebra:
	\begin{align}\label{simphopf2}
	\xymatrix@C=40pt{
		\mathcal{H} \: = \: \: \ar@{.}[r] &
		{H}_{3}
		\ar@<3ex>[r] \ar@<2ex>[r] \ar@<1ex>[r] \ar@<0ex>[r] &
		{H}_{2}\ar@<3ex>[r]|{d_2} \ar@<1.5ex>[r]|{d_1} \ar@<0ex>[r]|{d_0}
		\ar@/^1pc/[l] \ar@/^1.5pc/[l] \ar@/^2pc/[l]&
		{H}_{1}\ar@<1.5ex>[r]|{d_1} \ar[r]|{d_0}
		\ar@/^1pc/[l]|{s_0} \ar@/^1.5pc/[l]|{s_1} &
		{H}_{0}\ar@/^1pc/[l]|{s_0} } 
	\end{align}
	By using Hopf kernels in \eqref{simphopf2}, we obtain a new  simplicial Hopf algebra where the first three components are:
	\begin{align}\label{idea}
	\xymatrix@R=40pt@C=35pt{
		\lk (d_{0}) \subset H_3 \ar@<3ex>[r]|{d_3} \ar@<1.5ex>[r]|{d_2} \ar@<0ex>[r]|{d_1}&
		{\lk (d_{0}) \subset H_2}\ar@<1.5ex>[r]|{d_2} \ar[r]|{d_1}
		\ar@<2ex>[l]|{s_1} \ar@<3.6ex>[l]|{s_2}&
		{\lk (d_{0}) \subset H_1}\ar@<2ex>[l]|{s_1}
	}
	\end{align}
	Note that, we used the restrictions of the face and degeneracy maps. This idea can be iterated. For instance, when we take Hopf kernels again in \eqref{idea}, we get a new simplicial Hopf algebra where the first three components are:
	\begin{align}\label{idea2}
	\xymatrix@R=40pt@C=25pt{
		\scalebox{0.93}{$ \lk (d_{0}) \cap \lk (d_{1})  \subset H_4 $} \ar@<3ex>[r]|{d_4} \ar@<1.5ex>[r]|{d_3} \ar@<0ex>[r]|{d_2}&
		\scalebox{0.93}{$ \lk (d_{0}) \cap \lk (d_{1})  \subset H_3 $} \ar@<1.5ex>[r]|{d_3} \ar[r]|{d_2}
		\ar@<2ex>[l]|{s_2} \ar@<3.6ex>[l]|{s_3}&
		\scalebox{0.93}{$ \lk (d_{0}) \cap \lk (d_{1})  \subset H_2 $} \ar@<2ex>[l]|{s_2} 
	}
	\end{align}
\end{deco}

   \begin{decomp}
	If we apply Theorem \ref{isomorphism} to $H_1$ in \eqref{simphopf2}, we get
	\begin{align*}
		\begin{split}
	H_{1} &\cong \lk (d_{0})\osma{0} s_{0}(H_{0}) \\
	&=NH_{1} \osma{0} s_{0}(NH_{0}).
		\end{split}
	\end{align*}
	Similarly, considering $H_2$, we first get
	\begin{align*}
		\begin{split}
	H_{2} &\cong \lk (d_{0}) \osma{0} s_{0}(H_{1}) \\
	&\cong \lk(d_{0}) \osma{0} s_{0} \big( NH_{1}\osma{0} s_{0}(NH_{0}) \big),
		\end{split}
	\end{align*}
	and by applying Theorem \eqref{isomorphism} in \eqref{idea}, we further have
	\begin{align*}
	\begin{split}
	H_{2} & \cong \big( \lk(d_{1})\mid _{\lk(d_{0})} \osma{1}
	s_{1}(\lk(d_{0})) \big) \osma{0} s_{0} \big( NH_{1} \osma{0} s_{0}(NH_{0}) \big) \\
	&=\big( (\lk(d_{1})\cap \lk(d_{0})) \osma{1}
	s_{1}(\lk(d_{0})) \big) \osma{0} s_{0} \big( NH_{1} \osma{0} s_{0}(NH_{0}) \big) \\
	&=\big( NH_{2} \osma{1} s_{1}(NH_{1}) \big) \osma{0} \big(
	s_{0}(NH_{1}) \osma{0} s_{1}s_{0}(NH_{0})\big).
	\end{split}
	\end{align*}	
	One level further, by using \eqref{idea2}, we \revt{obtain}
	\begin{align*}
		\begin{split}
	H_{3} \cong \, & \big( \left( NH_{3} \osma{2} s_{2}(NH_{2})\right)   \osma{1} (s_{1}( NH_{2} \osma{1}s_{2}s_{1}(NH_{1})))\big)  \\
	& \quad \quad \quad \quad \quad \osma{0} \big (s_{0} ( NH_{2}) \osma{1}
	s_{2}s_{0}(NH_{1}) \osma{0}\left(s_{1}s_{0}(NH_{1}) \osma{0}
	s_{2}s_{1}s_{0}(NH_{0})\right) \big).
		\end{split}
	\end{align*}
	By iteration, we get the general formula:
\end{decomp}

\begin{thm}[Simplicial Decomposition]
	Let $\mathcal{H}$ be a simplicial Hopf algebra.  We have the decomposition of $H_{n}$, for any $n \geq 0$ as follows:\,\footnote{For the sake of simplicity, each smash product $\tn_{\rho_i}$ is just written by $\tn$ in the formula.}
	\begin{align*}
	\begin{split}
	H_n \cong \Big( \cdots \big( NH_n \otimes & s_{n-1}NH_{n-1} \big) \otimes \cdots s_{n-2} \cdots s_1NH_1 \Big) \, \otimes \\
	& \quad \Big( \cdots \big( s_0 NH_{n-1} \otimes s_1 s_0 NH_{n-2} \big) \otimes \cdots \otimes s_{n-1} s_{n-2} \cdots s_0 NH_0 \Big).
	\end{split}
	\end{align*}
\end{thm}

\section{Iterated Peiffer Pairings}\label{iterated}

The following notation and terminology is derived from \citep{Car1,CC1} where it is used for both simplicial groups and simplicial algebras. 

\subsection{The poset of surjective maps}

For the ordered set $[n]=\{0 < 1 < \cdots < n \}$, let $\alpha_i=\alpha_{i}^{n}\colon [n+1] \to [n]$ be the \rev{increasing} surjective map given by
\begin{align*}
\alpha _{i}^{n}(j)=\left\{ 
\begin{tabular}{cc}
$j$ & if $j\leq i$ \\ 
$j-1$ & if $j>i$%
\end{tabular}%
\right. 
\end{align*}
Let $S(n,n-l)$ be the set of all \rev{monotone increasing} surjective maps from $[n]$ to $[n-l]$. 
\rev{This} can be generated from the various $\alpha_{i}^{n}$ by composition. The composition of these generating maps satisfies the property \rev{$\alpha_{j} \, \alpha_{i} = \alpha_{i-1} \, \alpha_{j}$} with the condition $j<i$. This implies that  each element $\alpha \in S(n,n-l)$ has a unique expression as \rev{$\alpha = \alpha_{{i}_{1}} \, \alpha_{{i}_{2}} \, \cdots \, \alpha_{{i}_{l}}$} with \rev{$0 \leq i_{1} < i_{2} < \dots <i_{l} \leq n$}, where the indices $i_{k}$ are the elements of $[n]$ at which $\alpha(i)=\alpha(i+1)$. Clearly\rev{,} $S(n,n-l)$ is canonically isomorphic to the set $\{ (i_{l}, \cdots ,i_{1}) \colon 0 \leq i_{i} < i_{2} < \cdots < i_{l} \leq n-1 \}$.  For instance, the single element of $S(n,n)$ defined by the identity map on $[n]$, corresponds to the empty $0$-tuple $( \,\, )$ denoted by $\emptyset_{n}$. Similarly the only elements of $S(n,0)$ is $(n-1,n-2,\cdots,0)$. For all $n \geq 0$, let:
\begin{align*}
S(n) = \underset{0\leq l\leq n}{\bigcup }S(n,n-l).
\end{align*}
Any element of $S(n)$ is of the form \rev{$(i_{l},\cdots,i_{1})$ where $0 \leq i_{1} < i_{2} < \dots <i_{l} < n$}.  If $\alpha = (i_{l},\cdots,i_{1})$, then we say $\alpha$ has length $l$ and will write $\#\alpha = l$.
Consider the lexicographic order on $S(n)$. \rev{That is to say, $\alpha = (i_{l},\cdots,i_{1}) < \beta = (j_{m},\cdots,j_{1})$ in $S(n)$,
\begin{center}
	if $i_1 = j_1 , \cdots , i_k = j_k $ but $i_{k+1} > j_{k+1}$ ($k>0$) 
\end{center}
or
\begin{center}
	if $i_1 = j_1, \cdots, i_l = j_l$ and $l < m$ ,
\end{center}
that makes $S(n)$ a totally ordered set.} For instance, the orders of $S(2),S(3)$ and $S(4)$ are respectively:
\begin{align*}
S(2) = \,\, & \{ \emptyset_{2} < (1) < (0) < (1,0) \} \\
S(3) = \,\, & \{ \emptyset_{3} < (2) < (1) < (2,1) < (0) < (2,0) < (1,0) < (2,1,0) \} \\
S(4) = \,\, & \{ \emptyset_{4} < (3) < (2) < (3,2) < (1) < (3,1) < (2,1) < (3,2,1) < (0) <(3,0) \\ & \q\q\q\q\q < (2,0) < (3,2,0) < (1,0) < (3,1,0) < (2,1,0) < (3,2,1,0) \}
\end{align*}
If $\alpha,\beta \in S(n)$, we define $\alpha \cap \beta$ to be the set of indices which belong to both $\alpha$ and $\beta$.

\subsection{Iterated Peiffer pairings}

Let $\mathcal{H}$ be a simplicial Hopf algebra and \revt{$(NH_{\bullet},\d_{\bullet})$} be its Moore complex. We define the set $P(n)$ consisting of pairs of elements $(\alpha,\beta)$ from $S(n)$ with $\alpha \cap \beta = \emptyset$ and $\beta < \alpha$, with respect to lexicographic ordering in $S(n)$ where $\alpha=(i_{l},\cdots,i_{1})$, $\beta=(j_{m},\cdots,j_{1}) \in S(n)$. The pairings
\begin{align*}
\{F_{\alpha ,\beta }:NH_{n-\#\alpha } \rev{\,\otimes\,} NH_{n-\#\beta }\to NH_{n}\mid
(\alpha ,\beta )\in P(n),n\geq 0\}
\end{align*}
are defined as composites in the diagram:
\begin{align}\label{peifferdiagram}
\xymatrix@R=40pt@C=40pt{
	NH_{n-\#\alpha } \rev{\,\otimes\,} NH_{n-\#\beta } \ar[r]^-{F_{\alpha ,\beta }} \ar[d]_{s_{\alpha} \rev{\,\otimes\,} s_{\beta}}
	& NH_{n} \\
	H_{n} \rev{\,\otimes\,} H_{n} \ar[r]^{\ad}
	& H_{n} \ar[u]_{f_{n}}
}
\end{align}
such that
\begin{align*}
s_{\alpha }=s_{i_{l}}\ldots s_{i_{1}}:NH_{n-\#\alpha }\rightarrow H_{n},\ \
\ s_{\beta }=s_{j_{m}}\ldots s_{j_{1}}:NH_{n-\#\beta }\rightarrow H_{n},
\end{align*}
and  $f_n \colon H_{n}\rightarrow NH_{n}$ is defined by the composition $f_n=f_{n-1} \ldots f_1 f_{0}$, where $f_i$ is the Hopf kernel generator map defined in \rev{the proof of} Theorem \ref{isomorphism}.
%\begin{align*}
%p_{i} \colon x \longmapsto \underset{(x)}{\sum } \, x' \, s_{i}d_{i}S(x'') \in RKer(d_i)  
%\end{align*}

\begin{remark}
	In the prerequisite for the above idea, one can also consider the case $\alpha > \beta$ which creates same type of Peiffer elements, namely the same elements under the antipode. Moreover, lexicographic order guarantees the Peiffer elements to be non-trivial in the sense of simplicial identities. 
\end{remark}

\subsection{Calculating Peiffer pairings} 

In this subsection, we obtain the Peiffer pairings that needed in the sequel.

\subsubsection{$n=2$ case}\label{1higher-peiffer}

We have unique type of element for this case, by taking $\alpha =(0)$, $\beta =(1)$. Then $F_{(0)(1)}(x,y) \in NH_{2}$ is calculated for all $x,y\in NH_{1} $ as follows:
\begin{align}
F_{(0)(1)}(x,y)& =f_{1}f_{0} \, \big( s_{0}(x)\vartriangleright _{ad}s_{1}(y)%
\big) \nonumber \\
& =f_{1}\left( \underset{(s_{0}(x)\vartriangleright _{ad}s_{1}(x))}{\sum }%
\big( s_{0}(x)\vartriangleright _{ad}s_{1}(y)\big) ^{\prime
} \, s_{0}d_{0} \rev{\Big( S \big( s_{0}(x)\vartriangleright _{ad}s_{1}(y) \big)^{\prime \prime } \Big) }%
\right) \nonumber \\
& =f_{1}\left( \underset{(x)(y)}{\sum }\big( s_{0}(x^{\prime
})\vartriangleright _{ad}s_{1}(y^{\prime })\big) \, S \rev{\Big( s_{0}d_{0}\big(
s_{0}(x^{\prime \prime })\vartriangleright _{ad} s_{1}(y^{\prime \prime
}) \big) \Big) } \right) \nonumber \\
& =f_{1}\left( \underset{(x)(y)}{\sum }\big( s_{0}(x^{\prime
})\vartriangleright _{ad}s_{1}(y^{\prime })\big) \, S \big( s_{0}(x^{\prime
\prime })\vartriangleright _{ad}s_{0}s_{0}d_{0}(y^{\prime \prime })%
\big) \right) \nonumber \\
& =f_{1}\left( \underset{(x)}{\sum } \, \big( s_{0}(x^{\prime
})\vartriangleright _{ad}s_{1}(y)\big) \, S \big( s_{0}(x^{\prime \prime
})\vartriangleright _{ad}1_{H_2}\big) \right) \nonumber \\
& =f_{1} \, \big( s_{0}(x)\vartriangleright _{ad}s_{1}(y) \big) \nonumber \\
& =\underset{(x)(y)}{\sum }\big( s_{0}(x^{\prime })\vartriangleright
_{ad}s_{1}(y^{\prime })\big) \, s_{1}d_{1}\Big( S \big( s_{0}(x^{\prime \prime
})\vartriangleright _{ad} s_{1}(y^{\prime \prime }) \big) \Big) \nonumber \\
& =\underset{(x)(y)}{\sum }\big( s_{0}(x^{\prime })\vartriangleright
_{ad}s_{1}(y^{\prime })\big) \, S \big( s_{1}(x^{\prime \prime
})\vartriangleright _{ad} s_{1}(y^{\prime \prime })\big)  \label{F01}
\end{align}%

\rev{Remark that since $y \in NH_1 = \lk (d_0)$,  we write $\sum_{(y)} y' \tn d_0(y'') = y \tn 1_{H_0}$ follows from \eqref{hopfkernel}, and this implies $\sum_{(y)} s_1(y') \tn s_0 s_0 d_0(y'') = s_1(y) \tn 1_{H_2}$.}

\subsubsection{$n=3$ case}\label{sec:higher-peiffer}

The possible six Peiffer elements belonging to $NH_{3}$ are:
\begin{align*}
F_{(1,0)(2)} \, , \,\, F_{(2,0)(1)} \, , \,\, F_{(0)(2,1)} \, , \,\, 
F_{(0)(1)} \, , \,\,  F_{(0)(2)} \, , \,\,  F_{(1)(2)} \, .
\end{align*}
If we calculate these elements, we get:

\medskip

a) for all $x \in NH_{1}$ and $y \in NH_{2}$, 
\begin{itemize}
	\item $F_{(1,0)(2)}(x,y) = \displaystyle\sum_{(x)(y)} \big( s_{1}s_{0}(x^{\prime })\vartriangleright
	_{ad}s_{2}(y^{\prime })\big) \, S \big( s_{2}s_{0}(x^{\prime \prime
	})\vartriangleright _{ad} s_{2}(y^{\prime \prime })\big) $
	\item $F_{(2,0)(1)}(x,y) $ \begin{multline*} = \underset{(x)(y)}{\sum } \big( s_{2}s_{0}(x')\vartriangleright_{ad}s_{1}(y')\big) \, S \big( s_{2}s_{1}(x'')\vartriangleright _{ad} s_{1}(y'')\big) \\ S \Big( \big( s_{2}s_{0}(x''') \vartriangleright_{ad} s_{2}(y''') \big) \, S \big( s_{2}s_{1}(x'''') \vartriangleright_{ad} s_{2}(y'''') \big) \Big) \end{multline*}
\end{itemize}

b) for all $x \in NH_{2}$ and $y \in NH_{1}$,
\begin{itemize}
	\item $F_{(0)(2,1)}(x,y)$ \begin{multline*} = \underset{(x)(y)}{\sum }\big( s_{0}(x^{\prime })\vartriangleright
	_{ad}s_{2}s_{1}(y^{\prime })\big) \, S \big( s_{1}(x^{\prime \prime
	})\vartriangleright _{ad} s_{2}s_{1}(y^{\prime \prime })\big) \\ S \Big( s_{2}s_{1}(y''') \, S \big( s_{2}(x''') \vartriangleright_{ad} s_{2}s_{1}(y'''') \big) \Big) \end{multline*}
\end{itemize}

c) for all $x,y \in NH_{2}$,
\begin{itemize}
	\item $F_{(0)(1)}(x,y)$ \begin{multline*} = \underset{(x)(y)}{\sum }\big( s_{0}(x^{\prime })\vartriangleright
	_{ad}s_{1}(y^{\prime }) \big) \, S \big( s_{1}(x^{\prime \prime
	})\vartriangleright _{ad} s_{1}(y^{\prime \prime }) \big) \, S \Big( s_{2}(y''') \, S \big( s_{2}(x''') \vartriangleright_{ad} s_{2}(y'''') \big) \Big) \end{multline*} 
	\item $F_{(0)(2)}(x,y) = \ \displaystyle\sum_{(y)} \ \big( s_{0}(x)\vartriangleright_{ad}s_{2}(y') \big) \, S \big( s_{2}(y'') \big) $
	\item $F_{(1)(2)}(x,y) = \displaystyle\sum_{(x)(y)} \big( s_{1}(x^{\prime })\vartriangleright
	_{ad}s_{2}(y^{\prime })\big) \, S \big( s_{2}(x^{\prime \prime
	})\vartriangleright _{ad} s_{2}(y^{\prime \prime }) \big) $
\end{itemize}

\section{More on Crossed Modules}

\edit{In this section, we deal with the functorial relationship between the categories of crossed modules \rev{of Hopf algebras} and simplicial Hopf algebras.}

\subsection{From simplicial Hopf algebras to crossed modules}

\begin{thm}\label{simp-xmod}
	Let $\H$ be a simplicial Hopf algebra with \rev{a} Moore complex of length one.  We have the crossed module
	\begin{align*}
	\d_1 \colon  NH_{1} \to  H_0,
	\end{align*}  
	where the action $\rh \colon H_0 \otimes NH_1 \to NH_1$ is defined by
	\begin{align*}
	k \rh x = s_0 (k) \ad x ,
	\end{align*}
	for all $k \in H_0, \, x \in NH_1$.
	
	\medskip
	
	\rev{This construction yields a functor $\mathrm{X_1 \colon \{SimpHopf_{\leq 1}^{cc} \} \to \{XHopf^{cc}\}}.$}
\end{thm}

\begin{proof}
	\rev{The first crossed module condition ($\d_1 (k \rh x) = k \ad \d_1 (x)$, $\forall k \in H_0$ and $\forall x \in NH_1$) follows immediately. Thus, let us prove the second one ($\d_1 (x) \rh y = x \ad y$, $\forall x,y \in NH_1$). Since the length of the Moore complex of $\H$ is one, we straightforwardly have $F_{(0)(1)}(x,y) \in NH_2 = \k$. Therefore, the restriction of $d_2 \colon H_2 \to H_1$ to $NH_2$ (i.e. $\d_2$) becomes the zero morphism. Consequently we obtain} 
	\begin{align*}
	d_2 \big( F_{(0)(1)}(x,y) \big) = \e(x) \e(y) 1_{H_1} \, .  
	\end{align*}
	\rev{And, following from \eqref{F01}, we can write}
	\begin{align}\label{trick0}
	\underset{(x)(y)}{\sum } ( d_2 s_{0}(x^{\prime })\vartriangleright
	_{ad} y' ) \, S ( x'' \ad y'' ) = \e(x) \e(y) 1_{H_1} \, ,
	\end{align}
	\rev{that yields}
	\begin{align*}
	d_2 s_{0} (x) \ad y = x \aa y \, ,
	\end{align*}
	\rev{through Appendix \ref{appendix:dia-appr}. Finally, we have}
	\begin{align*}
	\d_1 (x) \rh y & = s_0 d_1 (x) \ad y \\
	& = d_2 s_0 (x) \ad y \\
	& = x \ad y \, ,
	\end{align*}
	\rev{for all $x,y \in NH_1$, and complete the proof}.
\end{proof}

%\begin{proposition}
%	This construction yields a functor $\mathrm{X_1 \colon SimpHopf_{\leq 1} \to XHopf}.$
%\end{proposition}

%\bigskip
%
%The above functors yield the following:
%
%\begin{thm}
%	The category of crossed modules of Hopf algebras is naturally equivalent to the category of Simplicial Hopf algebras with Moore complex length $1$. This can be pictured as:
%	\begin{align*}
%	\xymatrix@R=20pt@C=10pt{
%		& \mathbf{Tr_2 Simp} \ar[dl]_{cos_2} & \\
%		\mathbf{Simp_{\leq 1}} \ar[rr]_{X_1}
%		& & \mathbf{XHopf} \ar[lu]_{( \,\, )_2}
%	}
%	\end{align*}
%\end{thm}

\subsection{\edit{From crossed modules to simplicial Hopf algebras}}

\n The following lemma is an obvious analogue to the group and Lie algebra versions given in \citep{C2,E1}. 

\begin{lem}\label{generalcase}
	For a given $n$-truncated simplicial Hopf algebra $\H_{\mid n}$, the length of the Moore complex of $\mathrm{cosk_n}(\H_{\mid n})$ is $n+1$. Moreover, we have:
	\begin{itemize}
		\item $N\big( \mathrm{cosk_n} (\H_{\mid n}) \big)_{n+1} \cong \mathrm{HKer} \big( \d_{n} \colon {N(\H_{\mid n})}_n \to {N(\H_{\mid n} )}_{n-1} \big)$,
		\item $\d_{n+1} \colon N\big( \mathrm{cosk_n} (\H_{\mid n}) \big)_{n+1} \to N\big( \mathrm{cosk_n} (\H_{\mid n}) \big)_{n} $ is injective.
	\end{itemize} 
\end{lem}

\begin{theorem}\label{xmod-equivalence}
	The category of crossed modules of Hopf algebras is equivalent to the category of simplicial Hopf algebras with Moore complex of length one.
\end{theorem}

\begin{proof}

Let $\d \colon I \to H$ be a crossed module. Put $\mathfrak{H}_0 = H$. From Definition \ref{tensor}, we can define the smash product Hopf algebra $\mathfrak{H}_1 = I \otimes_{\rho} H$ 
%where each element $(h \otimes x) \in H \otimes I$ can be visualized in the following simplicial form:
%\begin{equation}\label{line1}
%\xymatrix{\e (h) x \ar[r]^{h} & \partial (h) x }
%\end{equation}
%This notation is useful in practice as we have 
with two face $d_0 , d_1 \colon \mathfrak{H}_1 \to \mathfrak{H}_0$ and one degeneracy $s_0 \colon \mathfrak{H}_0 \to \mathfrak{H}_1$ maps:
%\begin{align*}
%	d_0 \Big( \xymatrix{\e (h) x \ar[r]^{h} & \partial (h) x }  \Big)  = \epsilon (h)x \, , \quad
%	d_1 \Big( \xymatrix{\e (h) x \ar[r]^{h} & \partial (h) x }  \Big)  = \partial (h)x,
%\end{align*}
\begin{align}\label{face1}
d_0 \big( u \tn x \big)  = \epsilon (u)x \, , \quad
d_1\big( u \tn x \big)   = \partial (u)x , \quad 
s_0 (x) = \big( 1_I \tn x \big) \, .
\end{align}
%and also a degeneracy $s_0 \colon \mathfrak{H}_0 \to \mathfrak{H}_1$ such that
%\begin{align*}
%	s_0 (x) = (1\otimes x) = \Big( \xymatrix{x \ar[r]^{1} & x } \Big) \, .
%\end{align*}

Also, there exists an action of $I \otimes_{\rho} H$ on $I$ given via (for all $u,v \in I$ and $x\in H$):
\begin{align*}
(u \otimes x) \act_{\star } v =(\partial (u) \, x) \rh v \, ,
\end{align*}
that we define $\mathfrak{H}_2 = I \otimes_{\star} (I \otimes_{\rho} H)$. Then we obtain three face maps  $d_0 , d_1 , d_2 \colon \mathfrak{H}_2 \to \mathfrak{H}_1$ given \revt{by}
\begin{align}\label{face2}
d_{0} \big( u \tn v \tn x \big)  = \big( \epsilon (u)v \otimes x \big) \, , \, d_{1} \big( u \tn v \tn x \big)  =  \big( uv \otimes x \big) \, , \,
d_{2} \big( u \tn v \tn x \big) = \big( u \otimes \partial (v)x \big) \, ,
\end{align}
%Any element of the 2-simplex can be visualized by
%
%\begin{equation}\label{triangle1}
%\xymatrix@R=20pt@C=20pt{& & \d (hk) x \\ & & \quad \quad \quad \\
%		\e(hk) x \ar[rruu]^{hk} \ar[rr]_{\e(h) k} & & {\d \big( \e(h) k \big) x} \ar[uu]_{h}}
%\end{equation}
%By using the triangle notation, we obtain three non-trivial face maps $d_0 , d_1 , d_2 \colon \mathfrak{H}_2 \to \mathfrak{H}_1$ as being
%\begin{align*}
%	d_{0} ( \mathfrak{H}_2 )  & = \Big( \xymatrix{\e (hk) x \ar[r]^-{\e(h) k} & \partial \big( \e(h) k \big) x }  \Big) =\epsilon (h)k\otimes x \\
%	d_{1} ( \mathfrak{H}_2 )  & = \Big( \xymatrix{\e (hk) x \ar[r]^{hk} & \partial (h k) x }  \Big) = hk\otimes x \\
%	d_{2} ( \mathfrak{H}_2 )  & = \Big( \xymatrix{ \partial \big( \e(h) k \big) x \ar[r]^{h} & \partial (h k) x }  \Big) = h\otimes \partial (k)x
%\end{align*}
and also two degeneracies $s_0 , s_1 \colon \mathfrak{H}_1 \to \mathfrak{H}_2$ given \revt{by}
\begin{align*}
s_{0}(u \otimes x) =(1_I \otimes u \otimes x), \quad
s_{1}(u \otimes x) =(u \otimes 1_I \otimes x) \, .
\end{align*}
%Remark that, we used some Hopf algebra properties to simplify above notations to obtain the corresponding face morphisms from 2-simplex \eqref{triangle1} regarding to 1-simplex \eqref{line1}, there are some Hopf algebra map properties needed to use. For instance, $\e (hk) = \e \big( \e(h) k \big)$ and so on.

\medskip 

Thus, we have a 2-truncated simplicial Hopf algebra $\mathfrak{H}_{\mid 2}$ as follows:
\begin{align*}
\xymatrix@R=50pt@C=50pt{
	I \otimes_{\star} (I \otimes_{\rho} H) \ar@<3ex>[r]|-{d_2} \ar@<1.5ex>[r]|-{d_1} \ar@<0ex>[r]|-{d_0}&
	{I \otimes_{\rho} H} \ar@<1.5ex>[r]|-{d_1} \ar[r]|-{d_0}
	\ar@<1.5ex>[l]|-{s_0} \ar@<3ex>[l]|-{s_1}&
	H \ar@<1.5ex>[l]|-{s_0}
}
\end{align*}
which is associated with the crossed module $\d \colon I \to H$.

\bigskip

Now, consider the simplicial Hopf algebra ${\H'} = \mathrm{cosk_2} (\mathfrak{H_{\mid 2}})$.  The length of the Moore complex of $\H'$ is basically three, from Lemma \eqref{generalcase}. In other words, $N(\H')$ has the form:
\begin{align*}
\cdots \overset{}{\longrightarrow } \k \overset{}{\longrightarrow } \lk (\d_2) \overset{}{\longrightarrow } NH' _2  \overset{\partial _{2}}{	\longrightarrow } NH' _1 \overset{\partial _{1}}{	\longrightarrow } NH' _0
\end{align*}
It is clear that $NH'_1=I$ and $NH'_0=H$ by definition. However, when we calculate $NH'_2$, we see that it is the zero object, namely $\k$. This implies that $NH'_3 \cong \lk (\d_2)$ is also $\k$,  
%Therefore, $N(H')$ becomes:
%\begin{align*}
%\cdots \overset{}{\longrightarrow }  \k \overset{}{\longrightarrow } \k  \overset{}{	\longrightarrow } NH' _1 \overset{\partial _{1}}{	\longrightarrow } H' _0
%\end{align*}
that means the Moore complex of $H'$ is of length one\footnote{Here we need to obtain 2-truncated simplicial Hopf algebra. If we stop construction at 1-truncation level and apply $\mathrm{cosk_1}$, we will get a simplicial Hopf algebra with Moore complex length two, and there will be no cancellation.}. Therefore, we obtain a functor $\mathrm{G_1 \colon \{XHopf^{cc}\} \to \{SimpHopf_{\leq 1}^{cc}\}}$.

\medskip

Consequently, it is immediate from the face morphisms given in \eqref{face1}, \eqref{face2} that, $\mathrm{X_1 G_1 \cong 1_{XHopf^{cc}}}$. On the other hand, considering the simplicial decomposition given in \rev{Theorem \ref{isomorphism}}, we have $\mathrm{G_1 X_1 \cong 1_{SimpHopf_{\leq 1}^{cc}}}$.
\end{proof}
%Therefore, we have proved \rev{the equivalence}

% \medskip
% 
%  it is immediate from the simplicial decomposition given in \S ? that, we have $\mathrm{X_1 G_1 \cong 1_{XMod}}$. On the other hand, by using the adjunction properties, we have $\mathrm{G_1 X_1 \cong 1_{SimpHopf_{\leq 1}}}$. So we immediately have the following result:

\begin{rem}
	An alternative construction of the functor $\mathrm{G_1}$ can be given due to \citep[\S 2.3.5 and \S 3.7.4]{Menagerie}. Briefly, for a given crossed module $X := \big( I  \overset{\partial}{	\longrightarrow } H \big)$, we obtain its nerve $K(X)$ as follows:
	\begin{itemize}
		\item $K(X)_0 = H $
		\item $K(X)_1 = I \tn_{\rev{\rho_1}} H$ 
		\item $K(X)_2 = I \tn_{\rev{\rho_2}} (I \tn_{\rev{\rho_1}} H)$
	\end{itemize}
	(with the same face and degeneracy maps \eqref{face1} and \eqref{face2}) 
	% before. $d_0,d_1 \colon K(X)_1 \to K(X)_0$ are: $$d_0(h \tn x)=\e(h)x, \quad d_1(h \tn x) = \d(h)x$$ and
	%	also $d_0,d_1,d_2 \colon K(X)_2 \to K(X)_1$ are: $$d_0(h \tn k \tn x)=\e(h)k \tn x, \quad d_1(h \tn k \tn x) = hk \tn x, \quad d_2(h \tn k \tn x) = h \tn \d(k) x \, ,$$
	in low dimensions; and the construction continues with: 
	\begin{align*}
	K(X)_n & = I \tn_{\rev{\rho_{n}}} K(X)_{n-1} \\ &= I \tn_{\rev{\rho_{n}}} ( \cdots (I \tn_{\rev{\rho_2}} (I \tn_{\rev{\rho_1}} H ) ) \cdots )
	\end{align*}
	%namely $K(X)_n$ consisting $n$-copies of $H$ 
	for all dimensions $n \geq 1$. Here $K(X)_n$ \rev{includes $n$-copies of $I$ that} acts on $I$ via the unique composed face map $\rho_{n+1}$, \rev{i.e. $\rho_1 = \rho$ and $\rho_2 = \star$ in the context of Theorem \eqref{xmod-equivalence}}. Note that $K(X)_n$ is isomorphic to $H_n$ in the simplicial decomposition form when put $NH_n = \k$, for all $n \geq 2$ in the decomposition formula. The equivalence of these two alternative proofs comes from the well-known categorical property that, the nerve of a category is 2-coskeletal. 
	%(that means the unit of the adjunction is an isomorphism).
\end{rem}

\section{2-Crossed Modules}

For completeness, we first recall the 2-crossed modules of groups and of Lie algebras. We note that, a lot of different conventions appear in the literature for 2-crossed modules. We follow Conduch\'{e}, \citep{C2} in defining group 2-crossed modules, and \rev{as derived from that, the definition} in \citep{JFM2} of 2-crossed module of Lie algebras (also called differential 2-crossed modules). 

\subsection{2-crossed modules of groups, and Lie algebras}

%\begin{definition}
%	A crossed module of groups \citep{Br,FM,BL} is given by a group homomorphism $\d\colon E \to G$, together with an action $\tr$ of $G$ on $E$, by automorphisms, such that the following Peiffer relations hold for each $e,f \in E$ and each $g \in G$:
%	\begin{itemize}
%		\item $\d(g \tr e) =g \,\d(e) \, g^{-1}$
%		\item $\d(e)  \tr f=e\, f \, e^{-1}$.
%	\end{itemize}
%	If the second Peiffer identity does not necessarily hold, then $(\d\colon E \to G,\tr)$ is called a pre-crossed module.  In such case, given $e,f \in E$, the element $\langle e,f \rangle = efe^{-1} \, (\d(e)  \tr f^{-1})$ is called the ``Peiffer commutator" of $e$ and $f$. Clearly, a pre-crossed module is a crossed module if, and only if, $\langle e,f \rangle = 1_E$, for each $e,f \in E$. Note also that Peiffer commutators are conjugation invariant: given $g$ in $G$ we have: $g \tr \langle e,f \rangle =\langle g \tr e, g \tr f\rangle$, for each $e,f \in E$. 
%	
%\end{definition}

\begin{definition}\label{2xgrp}
	A 2-crossed module of groups is given by a chain complex of groups
	\begin{equation*}
	L \overset{\partial _{2}}{\longrightarrow } E \overset{\partial _{1}}{%
		\longrightarrow } G
	\end{equation*}
	together with left actions $\vartriangleright $ of $G$ on $E,L$; and with a $G$-equivariant\,\footnote{For groups, $G$-equivariance means $g \tr \{e,f\}= \{g\tr e , g \tr f\}$, $\forall g \in G$ and $e,f \in E$.} bilinear map called Peiffer lifting
	\begin{equation*}
	\{\ , \ \}: E \times _{G} E \to L
	\end{equation*} 
	satisfying the following axioms, for all $l,m\in L$ and $e,f,g \in E$:
	\begin{description}
		\item[1)] $L \overset{\partial _{2}}{\longrightarrow } E \overset{\partial _{1}}{%
			\longrightarrow } G$ is a complex of $G$-modules \rev{where} $G$ acts on itself by conjugation,
		
		\item[2)] $\partial _{2}\{e , f\} = (efe^{-1}) \, (\d_1(e) \tr f^{-1})$,
		
		\item[3)] $\{\partial _{2}(l) , \partial _{2}(m)\}= lml^{-1}m^{-1}$,
		
		\item[4)]  $\big\{e , fg \big\} = \{ e,f\} \ \big( \d_1 (e) \tr f \big) \tr' \{e,g\}$,
		
		\item[5)]  $\big\{ ef, g \big\} = \{ e , fgf^{-1}\} \ \d_1(e) \tr \{f,g\} $,
		
		\item[6)] $ \big\{ \d_2(l) , e  \big\} \, \big\{ e,  \d_2(l)  \big\} = l \, \big( \d_1(e) \tr l^{-1} \big)  .$
	\end{description}
	In the fourth condition, we put the action:
	\begin{align}\label{actiongrp}
	e \tr' l =  l \, \big\{ \d_2(l^{-1}),e\big\} ,
	\end{align} 
	for each $e \in E$, $l \in L$, that turns $\big( \d_2\colon L \to E,\tr'\big)$ into a crossed module.
\end{definition}

\begin{definition}\label{2xlie}
	A 2-crossed module of Lie algebras (i.e. differential 2-crossed module) is given by a chain complex of Lie algebras
	\begin{equation*}
	\mathfrak{l} \overset{\partial _{2}}{\longrightarrow } \mathfrak{e} \overset{\partial _{1}}{%
		\longrightarrow } \mathfrak{g}
	\end{equation*}
	together with left actions $\vartriangleright $ of $\mathfrak{g}$ on $\mathfrak{e,l}$; and with a $\mathfrak{g}$-equivariant\,\footnote{For Lie algebras, $\mathfrak{g}$-equivariance means $g \tr \{u , v\}= \{g\tr u , v\}+\{u , g \tr  v\}$, $\forall g \in \mathfrak{g}$ and $u,v \in \mathfrak{e}$.} bilinear map called Peiffer lifting
	\begin{equation*}
	\{\ , \ \}: \mathfrak{e} \times_{\mathfrak{g}} \mathfrak{e} \to  \mathfrak{l}
	\end{equation*} 
	satisfying the following axioms, for all $x,y \in \mathfrak{l}$ and $u,v,w \in \mathfrak{e}$:
	\begin{description}
		\item[1)] $\mathfrak{l} \overset{\partial _{2}}{\longrightarrow } \mathfrak{e} \overset{\partial _{1}}{%
			\longrightarrow } \mathfrak{g}$ is a complex of $\mathfrak{g}$-modules \rev{where} $\mathfrak{g}$ acts on itself by adjoint representation,
		
		\item[2)] $\partial _{2}\{u , v\} = [u,v] - \d_1(u) \tr v $,
		
		\item[3)] $\{\partial _{2}(x) , \partial _{2}(y)\}= [x,y]$,
		
		\item[4)]  $\big\{u , [v,w] \big\} = \big\{ \d_2 \{u , v \} , w \big\} - \big\{ \d_2 \{u , w \} , v \big\} $,
		
		\item[5)]  $\big\{ [u,v] , w \big\} = \d_1 (u) \tr \big\{ v,w  \big\} + \big\{ u , [v,w]  \big\} - \d_1 (v) \tr \big\{ u,w  \big\} -  \big\{ v , [u,w]  \big\} $,
		
		\item[6)] $ \big\{ \d_2(x) , v  \big\} + \big\{ v,  \d_2(x)  \big\} = -\d_1 (v) \tr x  .$
	\end{description}	
	When we put:
	\begin{align}\label{actionlie}
	v \tr' x =  - \big\{ \d_2(x) , v  \big\}
	\end{align}
	for each $x \in \mathfrak{l}$, $v \in \mathfrak{e}$, that turns $\big( \d_2 \colon \mathfrak{l} \to \mathfrak{e},\tr' \big)$ into a differential crossed module.
\end{definition}

\subsection{2-crossed modules of Hopf algebras}

\begin{definition}
	A ``2-crossed module" of Hopf algebras is given by a chain complex
	\begin{equation*}
	K\overset{\partial _{2}}{\longrightarrow }I \overset{\partial _{1}}{%
		\longrightarrow } H
	\end{equation*}
	of Hopf algebras (i.e. $\partial_1 \partial_2$ is zero morphism) with the actions $\rh $ of $H$ on $I,K$, and also on itself by the adjoint action; together with an $H$-equivariant\,\footnote{Here $H$-equivariance means $a \rh \{x,y\} = \sum_{(a)} \{a' \rh x , a'' \rh y \}$, $\forall a \in H$ and $x,y \in I$.} bilinear map called Peiffer lifting
	\begin{equation*}
	\{\ , \ \}: I \times _{H} I \longrightarrow K
	\end{equation*} 
	satisfying the following axioms, for all $x,y,z \in I$ and  $k,l \in K$:
	\begin{description}
		\item[1)] $K\overset{\partial _{2}}{\longrightarrow }I \overset{\partial _{1}}{%
			\longrightarrow } H $ is a complex of $H$-module bialgebras,\footnote{Namely, $\d_2 (a \rh k) = a \rh \d_2 (k)$, and $\d_1 (a \rh x) = a \ad \d_1 (x)$, $\forall a \in H$, $x \in I$, $k \in K$.}
		
		\item[2)] $\partial _{2}\{x , y\} = \underset{(x)(y)}{\sum } (x' \vartriangleright_{ad} y') \,\, \partial_{1}(x'') \vartriangleright_{\rho} S(y'') $,
		
		\item[3)] $\{\partial _{2}(k) , \partial _{2}(l)\}= \ \underset{(l)}{\sum } \, (k \vartriangleright_{ad} l') \, S(l'') $\,\footnote{\rev{It is the commutator, i.e. $\sum_{(k)(l)} k'l'S(k')S(l')$.}},
		
		\item[4)] $\{x ,yz\} = \underset{(x)(y)}{\sum } \{x' , y'\} \ (\d_1(x'') \rh y'') \vartriangleright_{\rho}' \{x''' , z\} $,
		
		\item[5)] $\{xy , z\}= \underset{(x)(y)(z)}{\sum }  \{x',y' \aa z'\} \ \partial_{1}(x'') \arho \{y'',z''\}$,
		
		\item[6)] $ \underset{(k)(x)}{\sum } \{\partial _{2}(k') , x'\} \{x'' , \partial _{2}(k'')\} =  \underset{(k)}{\sum } \ k' \, \left( \partial_{1}(x) \arho S(k'') \right)$.
	\end{description}
	\rev{In the fourth condition, the action is:}
	\begin{align}\label{actionhopf}
	x \vartriangleright_{\rho}' k = \underset{(k)}{\sum } \, k' \, \{ \d_2 (S(k'')) , x \} \, ,
	\end{align}
	\rev{such that $\big( \d_2 \colon K \to I , \tr'_{\rho} \big)$ becomes a crossed module. Remark that $\d_1$ is only a precrossed module in general -- just like in the cases of groups and Lie algebras.}
\end{definition}

\begin{definition}
	\rev{A 2-crossed module morphism of Hopf algebras between   $K\overset{\partial _{2}}{\longrightarrow }I \overset{\partial _{1}}{\longrightarrow }H $ and $K' \overset{\partial' _{2}}{\longrightarrow } I' \overset{\partial'_{1}}{\longrightarrow }H'$ is given by a triple $(f_2,f_1,f_0)$ that consists of Hopf algebra maps $f_0 \colon H \to H'$, $f_1 \colon I \to I'$ and $f_2 \colon K \to K'$, making the diagram:
	\begin{equation*} 
	\xymatrix@R=20pt@C=50pt{
		K
		\ar[r]^{\partial _{2}}
		\ar[d]^{f_{2}}
		& I
		\ar[r]^{\partial _{1}}
		\ar[d]^{f_{1}}
		& H
		\ar[d]^{f_{0}}
		\\
		K^\prime
		\ar[r]_{\partial _{2}^{\prime }}
		& I^\prime
		\ar[r]_{\partial _{1}^{\prime }}
		& H^\prime
	}
	\end{equation*}
	commutative, and also preserving the actions of $H$ and $H'$, as well as the Peiffer liftings:}
	\begin{align*}
	f_{1}(a \rh x)&=f_{0}(a) \rh f_{1}(x),  \textrm{ for all }  a \in H  \textrm{ and } x \in I,
	\\
	f_{2}(a \rh k)&=f_{0}(a) \rh f_{2}(k), \textrm{ for all }  a \in H  \textrm{ and } k \in K,
	\\
	f_{2}\{x,y\}&=\{f_{1}(x) , f_{1}(y)\},  \textrm{ for all } x,y \in I. 
	\end{align*} 
\end{definition}

\rev{Therefore, we have defined the category of 2-crossed modules of Hopf algebras which will be denoted by $\mathrm{\{X_2 Hopf^{cc}\}}$}.

\begin{example}
	\edit{For a given crossed module $I \overset{\partial}{\longrightarrow } H$, we can obtain a natural 2-crossed module $\revt{\k} \rev{\,\overset{}{\longrightarrow }\,} I \overset{\partial _{}}{\longrightarrow } H$ \revt{where $\k$ is the zero object as usual}. This gives rise to an inclusion functor $\mathrm{\{XHopf^{cc}\} \to \{X_{2}Hopf^{cc}\}}$.}
\end{example}

\begin{example}
Let $I \overset{\partial_1}{\longrightarrow } H$ be a precrossed module. Then we have a 2-crossed module
\begin{align*}
	\lk (\d_1)\overset{\partial _{2}}{\longrightarrow }I\overset{\partial _{1}}{\longrightarrow }H \, ,
\end{align*}
where $\d_2$ is the inclusion map, and the Peiffer lifting is \rev{necessarily} given by (\rev{$\forall x,y \in I$}):
\begin{align*}
\{ x , y \} = 	 \underset{(x)(y)}{\sum } (x' \vartriangleright_{ad} y') \,\, \partial_{1}(x'') \vartriangleright_{\rho} S(y'') \, . 
\end{align*}	
\end{example}

\subsection{Coherence with $Prim$ and $Gl$}

\begin{thm}
	The functor $Prim$ preserves the 2-crossed module structure.
\end{thm}

\begin{proof}
	Suppose that $K\overset{\partial _{2}}{\longrightarrow }I\overset{\partial _{1}}{\longrightarrow }H$ is a 2-crossed modules of  Hopf algebras. We already know that the primitive elements preserve the actions \citep{JFM1}. Therefore, we have the complex of $Prim(H)$-modules of Lie algebras. Since the rest of the proof consists in routine calculations by setting $\delta (a)=1\tn a + a \tn 1$ and $Prim (a \ad b) = [a,b]$, we only prove that the derived crossed module action \rev{\eqref{actionhopf} is compatible with \eqref{actionlie}}, namely 
	\begin{align*}
	x \vartriangleright_{\rho}' k \, & = \, \underset{(k)}{\sum } \, k' \, \{ \d_2 (S(k'')) , x \} \\
	& = 1_K  \, \{ - \d_2 (k) , x \} + k  \, \{ \d_2 (1_K) , x \} \\
	& = - \{ \d_2 (k) , x \} \\
	& = x \tr' k  \, ,
	\end{align*}
	\rev{for all $x \in I$ and $k \in K$. Regarding the penultimate step of the calculation: from the fifth condition of the Peiffer lifting, we obtain $\{1_I,x\}=\{x,1_I\}= \e(x)1_K$. Moreover, since $x$ is primitive, we have $\e(x)=0$ that yields $\{1_I,x\}=\{x,1_I\}=0$.}
\end{proof}

\begin{thm}
	The functor $Gl$ preserves the 2-crossed module structure.  
\end{thm}

\begin{proof}
	Follows immediately by letting $\delta(x)=x \tn x$ and $Gl(x \ad y) = xyx^{-1}$.
\end{proof}

\begin{proposition}\label{final2}
	Consequently, we have the functors
	\begin{align*}
	\xymatrix@R=40pt@C=30pt{ \mathrm{\{X_2 Lie\}}  &
		\mathrm{\{X_2 Hopf^{cc}\}} \ar[r]^{Gl} \ar[l]_{Prim} &
		\mathrm{\{X_2 Grp\}} 
	}
	\end{align*}
	extending \eqref{functors} to the 2-crossed module level.
\end{proposition}

\subsection{From simplicial Hopf algebras to 2-crossed modules}

\begin{thm}\label{simp-2xmod}
	Let $\H$ be a simplicial Hopf algebra with \rev{a} Moore complex of length two. Consider:
	\begin{align}\label{2xchain}
	NH_{2} \ra{\d_{2}}  NH_{1} \ra{\d_{1}} H_0 \, .
	\end{align}
	Then we have the actions:
	\begin{itemize}
		\item[-] of $H_0$ on $NH_1$, given by \rev{$n \rh m = s_{0}\left( n\right) \ad m$, for all $n \in H_0$ and $m \in NH_1$},
		\item[-] of $H_0$ on $NH_2$, given by \rev{$n \rh l = s_{1}s_{0}\left(	n\right) \ad l $, for all $n \in H_0$ and $l \in NH_2$},
	\end{itemize}
	and the Peiffer lifting $\{\ , \ \}:  NH_1 \otimes_{H_0} NH_1 \to NH_2$ as being
	\begin{align*}
	\left\{x , y\right\} = \underset{(x)(y)}{\sum }\big( s_{1}(x')\vartriangleright
	_{ad}s_{1}(y')\big) \, S \big( s_{0}(x'')\vartriangleright _{ad} s_{1}(y'' )\big) ,
	\end{align*}
	\rev{for all $x,y \in NH_1$}, that makes \eqref{2xchain} into a 2-crossed module of Hopf algebras.
\end{thm}

\begin{proof}
	\rev{Since the length of the Moore complex of $\H$ is two, we have $\d_3 (x) = \e (x)\,1_{H_2}$, for all $x \in NH_3$ - which plays the key role in the whole proof. But first, let us let us examine some direct consequences that will be helpful in the calculations:}
	
	\begin{itemize}
		
	\item
	
	\revt{We already know that
		\begin{align*}
			\partial_{3} \left( F_{(1,0)(2)}(x,y) \right) = \e(x) \e(y) 1_{H_2},
		\end{align*}
		for all $x \in H$ and $y \in K$. If we calculate the left hand side, we obtain
		\begin{align*}
			\underset{(x)(y)}{\sum }\left( s_{1}s_{0}d_{1}(x') \aa y'\right) S \left( s_{0}(x'') \aa y'' \right) = \e(x) \e(y) 1_{H_2},
		\end{align*}
		which implies 
		\begin{align*}
			\left(s_{1}s_{0}d_{1}(x) \aa y\right) = \left(s_{0}(x) \aa y \right).
		\end{align*}
		Moreover, recalling the action of $H_0$ on $NH_2$ given above, we explicitly get
		\begin{align}\label{cond5}
			\d_1 (x) \rh k = \left(s_{1}s_{0}d_{1}(x) \aa k \right) = \left(s_{0}(x) \aa k \right) ,
		\end{align}
		for all $x \in NH_1$ and $k \in NH_2$.		
		}
	
	\item The action $x \rh' k$ given in \eqref{actionhopf} corresponds to $s_1(x) \aa k$, for all $x \in NH_1$ and $k \in NH_2$; which follows from $\partial_{3} \left( F_{(0)(2,1)}(k,x) \right) = \e(k) \e(x) 1_{H_2}$.
	
	\end{itemize}
	
	Now, let us use the Peiffer pairings obtained in \cref{sec:higher-peiffer} to prove the conditions:
	
	\medskip
	
	\n \textbf{1)} Follows from the definition of the Moore complex.
	
	\medskip
	
	\n \textbf{2)} We straightforwardly  have \rev{(for all $x,y \in NH_1$):}
	\begin{align*}
	\begin{split}
	\partial _{2}\{x , y \} & = d_{2} \left( \underset{(x)(y)}{\sum } \big( s_{1}(x')\vartriangleright
	_{ad}s_{1}(y')\big) \, S \big( s_{0}(x'')\vartriangleright _{ad} s_{1}(y'' )\big) \right) \\
	& = \underset{(x)(y)}{\sum } \big( d_{2}s_{1}(x^{\prime })\vartriangleright
	_{ad} d_{2}s_{1}(y^{\prime }) \big) \, S \big( d_{2}s_{0}(x^{\prime \prime
	})\vartriangleright _{ad} d_{2}s_{1}(y^{\prime \prime }) \big) \\
	& = \underset{(x)(y)}{\sum } ( x' \vartriangleright_{ad} y' ) \,  \big( s_{0}d_{1}(x^{\prime \prime
	})\vartriangleright _{ad} S(y^{\prime \prime })\big) \\
	& = \underset{(x)(y)}{\sum } (x' \vartriangleright_{ad} y') \,\, \partial_{1}(x'') \vartriangleright_{\rho} S(y'').
	\end{split}	
	\end{align*}

	\n \textbf{3)} We get \rev{(for all $k,l \in NH_2$):}
	\begin{align*}
	\{\partial _{2}(k) , \partial _{2}(l)\} & = \underset{(k)(l)}{\sum } \big( s_{1}d_{2}(k^{\prime })\vartriangleright
	_{ad}s_{1}d_{2}(l^{\prime }) \big) \, S \big( s_{0}d_{2}(k^{\prime \prime
	})\vartriangleright _{ad} s_{1}d_{2}(l^{\prime \prime }) \big) \\
	& \because \quad \text{the fact that:} \quad \rev{\partial_{3} \left( F_{(0)(1)}(k,l) \right) = \e(k) \e(l) 1_{H_2}} \\
	& = \underset{(l)}{\sum } \, (k \vartriangleright_{ad} l') \, S(l'') .
	\end{align*}
	
	\n \textbf{4)} \revt{We get (for all $x,y,z \in NH_1$):}
	\begin{align*}
	\{x, yz\}  & =  \underset{(x)(yz)}{\sum } \big( s_{1}(x)' \aa s_{1}(yz)' \big) \, S \big( s_{0}(x)'' \aa s_{1}(yz)'' \big) \\ 
	& \because \quad \text{\revt{for the calculations,} \rev{see Appendix \ref{appendix:4}}} \\
	& = \underset{(x)(y)}{\sum } \, \{x' , y'\} \ (\d_1(x'') \rh y'') \rh' \{x''' , z \}.
	\end{align*}

	\medskip
	
	\n \textbf{5)} Again \revt{for all $x,y,z \in NH_1$, we get}:
	\begin{align*}
	\{xy , z\}  & =  \underset{(xy)(z)}{\sum } \big( s_{1}(xy)' \aa s_{1}(z)' \big) \, S \big( s_{0}(xy)'' \aa s_{1}(z)'' \big) \\ 
	& \because \quad \text{\revt{for the calculations,} \rev{see Appendix \ref{appendix:5}}} \\
	& = \underset{(x)(y)(z)}{\sum } \{x',y' \aa z'\} \ \textcolor{black}{ d_{1}(x'') \arho \{y'',z''\} }.
	\end{align*}

	\n \textbf{6)} We \rev{first} have \rev{(for all $k \in NH_2$ and $x \in NH_1$)}:
	\begin{align*}
	\{\partial_{2}(k),x\} & = \underset{(k)(x)}{\sum } \big( s_{1}d_{2}(k') \aa s_{1}(x') \big) \, S \big( s_{0}d_{2}(k'') \aa s_{1}(x'') \big) \\
	&  \because \quad \text{the fact that:} \quad \partial_{3} \left( F_{(0)(2,1)}(k,x) \right) = \e(k) \e(x) 1_{H_2} \\
	& = \ \underset{(x)}{\sum} \ \big( k \ad s_1 (x') \big) \, S \big( s_1 (x'') \big) \\
	& = \ \underset{(k)}{\sum} \ k' \ \big(s_1(x) \ad S(k'') \big),
	\end{align*}
	and
	\begin{align*}
	\{x,\partial_{2}(k)\} & = \underset{(k)(x)}{\sum } \big( s_{1}(x') \aa s_{1}d_{2}(k') \big) \, S \big( s_{0}(x'') \aa s_{1}d_{2}(k'') \big) \\
	& \because \quad \text{the fact that:} \quad \partial_{3} \left( \rev{F_{(2,0)(1)}(x,k)} \right) = \e(x) \e(k) 1_{H_2} \\ 
	& = \underset{(k)(x)}{\sum }  \big( s_{1}(x') \aa k' \big) \, S \big(s_{0}(x'') \aa k'' \big),
	\end{align*}
	that \rev{implies}
	\begin{align*}
	\underset{(k)(x)}{\sum }  \{\partial _{2}(k') , x'\} & \{x'' , \partial _{2}(k'')\} \\ & =  \underset{(k)(x)}{\sum } k' \, \big(s_1(x') \ad S(k'') \big) \, \big(s_{1}(x'') \aa \rev{k'''} \big) \, S \big(s_{0}(x''') \aa k'''' \big) \\
	& \because \quad \text{by using} \, \eqref{cond5} \\
	& = \ \underset{(k)}{\sum } \ k'  \left( \partial_{1}(x) \arho S(k'') \right) .
	\end{align*}

%	Since $\partial_{3} \left( F_{(0)(2,1)}(x,y) \right) = \e(x) \e(y) 1_{H}$, the action of $NH_1$ on $NH_2$ via $s_1$ is exactly the action $\vartriangleright'_{\rho}$ given in \eqref{actionhopf}.
%	
%	\medskip
%	
%	Remark that, we explicitly use a proper form of the idea given in \eqref{trick}. For instance, when we are obtaining \eqref{cond5} from \eqref{cond6}.
\end{proof}

\begin{corollary}\label{final3}
	We therefore obtain a functor $\mathrm{X_2 \colon \{SimpHopf_{\leq 2}^{cc}\} \to \{X_{2}Hopf^{cc}\}}$.
\end{corollary}

\subsection{\edit{From 2-crossed modules to simplicial Hopf algebras}}

\begin{theorem}\label{final5}
	The category of 2-crossed modules of Hopf algebras is equivalent to the category of simplicial Hopf algebras with Moore complex of length two.
\end{theorem}

\begin{proof}
Let us fix an arbitrary 2-crossed module $K \overset{\partial_2}{	\longrightarrow } I \overset{\partial_1}{	\longrightarrow } H $. We already have
% its 2-nerve $K(X)$ defined\footnote{\textrm{Hint:} As for the crossed module case, simplicial decomposition of $H_n$ gives us the correct underlying vector space of the $K(X)_n$ when we put $\k$ for $NH_n$.} as follows:
$\mathfrak{H}_0 = H$ and $\mathfrak{H}_1 = I \otimes_{\rho} H$
%\begin{itemize}
%	\item $\mathfrak{H}_0 = I$
%	\item $\mathfrak{H}_1 = H \otimes_{\rho} I$
%\end{itemize}
with the face and degeneracy maps as given in the crossed module case.

%	\medskip
%	
%	Product is:
%	\begin{align*}
%	(h,x)(h_2,x_2)=\Big( h(x \act h_2),xx_2 \Big)
%	\end{align*}

\medskip

\rev{To improve the readability of the proof, we also fix the variables  $k,l,m,k_2,l_2,m_2 \in K$, $x,y,z,x_2,y_2,z_2 \in I$ and $a \in H$ from now on.}

\medskip

\rev{First of all, $\d_2$ part of the 2-crossed module yields the smash product Hopf algebra $K \otimes_{\rho'} I$ considering \revt{the action $\rh'$ of $I$ on $K$ defined in \eqref{actionhopf}.}} There exists an action of $I \otimes_{\rho} H$ on $K \otimes_{\rho'} I$ given \revt{by}
\begin{align*}
(y \tn a) \act_{\ast} (k_2 \tn x_2) & = \underset{(y)(a)(x_2)}{\sum } \Big( \big( \d_1(y') a' \rh k_2 \big) \, S\big( \{y'', a'' \rh x'_2\} \big) \, \tn \, y''' \ad (a''' \rh x''_2)  \Big)
\end{align*}
that yields the smash product Hopf algebra
\begin{align*}
\mathfrak{H}_2 = (K \otimes_{\rho'} I) \otimes_{\ast} (I \otimes_{\rho} H) \, ,
\end{align*}
with the face maps $d_0,d_1,d_2 \colon \mathfrak{H}_2 \to \mathfrak{H}_1$:
\begin{align*}
d_0(k \tn x \tn y \tn a) & =(\e(k)\e(x)y \tn a) \\
d_1(k \tn x \tn y \tn a)  & =(\e(k)xy \tn a) \\
d_2(k \tn x \tn y \tn a)  & = (\d_2(k)x \tn \d_1(y) a)
\end{align*}
and also with two natural degeneracy maps $s_0, s_1 \colon \mathfrak{H}_1 \to \mathfrak{H}_2$.
%\begin{align*}
%s_0(h^{\dagger} \tn x) & =(1 \tn 1 \tn h^{\dagger} \tn x) \\
%s_1(h^{\dagger} \tn x)  & =(1 \tn h^{\dagger} \tn 1 \tn x) 
%\end{align*}

\medskip

Naturally, $K \otimes_{\rho'} I$ acts on $K$ with
\begin{align*}
(k \tn x) \act_{\star} l = k \ad (x \rhp l) \, ,
\end{align*}
hence we can construct $K \otimes_{\star} (K \otimes_{\rho'}  I)$.

	\medskip
	
	Then $I$ and $H$ acts on $K \otimes_{\star} (K \otimes_{\rho'}  I)$ respectively with:
	\begin{multline*}
	y \act_{\dagger_1} (l_2 \otimes m_2 \otimes z_2)  \\ = \underset{(y)(z_2)}{\sum }	\Big( { \d_1(y') \rh l_2 } \, \tn \, { (\d_1(y'') \rh m_2) \, S\big( \{ y''' , {z_2}' \} \big) } \, \tn \, 	{ y'''' \ad {z_2}''  } \, \Big) \, ,
	\end{multline*}
	and
	\begin{align*}
	a \act_{\dagger_2} (l_2 \otimes m_2 \otimes z_2) & = \underset{(a)}{\sum }
	\Big( { a' \rh l_2  } \tn 
	{ a'' \rh m_2  } \tn
	{ a''' \rh z_2 } \Big) \, .
	\end{align*}
	Therefore, $I \otimes_{\rho} H$ acts on $K \otimes_{\star} (K \otimes_{\rho'}  I)$ as follows:
	\begin{align}\label{main1}
	(y \tn a) \act_{\dagger} (l_2 \otimes m_2 \otimes z_2) = y \act_{\dagger_1} \Big( a \act_{\dagger_2} \big( l_2 \otimes m_2 \otimes z_2 \big) \Big) \, .
	\end{align}
	
	\medskip
	
	\n On the other hand, we have an action of $K$ on  $K \otimes_{\star} (K \otimes_{\rho'}  I)$ given \revt{by}
	\begin{align*}
	k \act_{\ddagger_1} & (l_2 \otimes m_2 \otimes z_2) \\ &
	=	\underset{(k)( m_2)(z_2)}{\sum } \Big( l_2 \, S \big( \big\{ \d_2(k') , \d_2( m'_2 ) \, z'_2 \big\} \big)  \, 
	\tn \, ( k'' \ad {m''_2} ) \, \{ \d_2(k''') , {z''_2} \}  
	\tn { z'''_2 } \Big) \, .
	\end{align*}
	%Here, notice that $kk''_2 k^{-1} \, \{ \d_2(k) , h''_2 \} = kk''_2(h''_2 \act k^{-1})$. 
	Moreover, we have another nontrivial action of $I$ on $K \otimes_{\star} (K \otimes_{\rho'}  I)$ which is
	\begin{align*}
	x \act_{\ddagger_2} & (l_2 \otimes m_2 \otimes z_2)  \\ &
	= \underset{(x)(m_2)(z_2)}{\sum }
	\Big( \big( \d_1(x') \rh l_2 \big) \, S \big( \big\{ x'' , \d_2({m'_2}) {z'_2} \big\} \big)  \otimes x''' \rhp {m''_2}  \otimes { x'''' \ad {z''_2} } \Big) \, .
	\end{align*}
	Therefore, $K \otimes_{\rho'} I$ acts on $K \otimes_{\star} (K \otimes_{\rho'}  I)$ as follows:
	\begin{align}\label{main2}
	(k \tn x) \act_{\ddagger} (l_2 \otimes m_2 \otimes z_2) = k \act_{\ddagger_1} \Big( x \act_{\ddagger_2} \big( l_2 \otimes m_2 \otimes z_2 \big) \Big) \, .
	\end{align}
	Consequently, we have an action of $\mathfrak{H}_2 = (K \otimes_{\rho'} I) \otimes_{\ast} (I \otimes_{\rho} H)$ on $K \otimes_{\star} (K \otimes_{\rho'}  I)$ which is given by\footnote{Subscripts on the right-hand side are useful when we calculate the product of two elements of $\mathfrak{H}_3$.}
	\begin{align}\label{main}
	(k \tn x \tn y \tn a) \act_{\bullet} (l_2 \otimes m_2 \otimes z_2) = (k \tn x) \act_{\ddagger} \Big( \big(y \tn a \big) \act_{\dagger} \big( l_2 \otimes m_2 \otimes z_2 \big) \Big) \, ,
	\end{align}
	that yields the smash product Hopf algebra
	\begin{align*}
	\mathfrak{H}_3 = \Big( K \otimes_{\star} (K \otimes_{\rho'}  I) \Big) \otimes_{\bullet}  \Big( (K \otimes_{\rho'} I) \otimes_{\ast} (I\otimes_{\rho} H) \Big) \, ,
	\end{align*}
	with the face maps $d_0 , d_1 , d_2 , d_3 \colon \mathfrak{H}_3 \to \mathfrak{H}_2$:
	\begin{align*}
	d_0 \seven & = \Big( \e(lm) k \otimes \e(z)x \otimes y \otimes a \Big) \\
	d_1 \seven & =  \underset{(z)}{\sum } \Big( \e(l) \, \big( m \, (z' \act k) \big) \otimes z'' x \otimes y \otimes a \Big) \\
	d_2 \seven & = \Big( \e(k) lm \otimes z \otimes xy \otimes a \Big)  \\
	d_3 \seven & = \Big( l \otimes \d_2 (m) z \otimes \d_2 (k) x \otimes \d_1(y) a \Big)  
	\end{align*}
	and also with three natural degeneracies $s_0, s_1, s_2 \colon \mathfrak{H}_2 \to \mathfrak{H}_3$.
	
	\medskip
	
%	We therefore obtain a 3-truncated simplicial Hopf algebra $\mathfrak{H}^{\mid 3}$ 
%	denoted by:
%	\begin{align}\label{simp2}
%	\xymatrix@R=50pt@C=50pt{
%		*+[l]{\mathfrak{H}_3} \ar@<4.5ex>[r]|{d_3} \ar@<3ex>[r]|{d_2} \ar@<1.5ex>[r]|{d_1} \ar@<0ex>[r]|{d_0} & {\mathfrak{H}_2}\ar@<3ex>[r]|{d_2} \ar@<1.5ex>[r]|{d_1} \ar@<0ex>[r]|{d_0} \ar@/^1pc/[l]|{s_0} \ar@/^1.75pc/[l]|{s_1} \ar@/^2.5pc/[l]|{s_1} &
%		\mathfrak{H}_1 \ar@<1.5ex>[r]|{d_1} \ar[r]|{d_0}
%		\ar@/^1pc/[l]|{s_0} \ar@/^1.75pc/[l]|{s_1}&
%		\mathfrak{H}_0 \ar@/^1pc/[l]|{s_0}
%	}
%	\end{align}
	
%	\medskip
We therefore obtain a 3-truncated simplicial Hopf algebra. Analogous to the crossed module case, considering its image under the $ \mathrm{cosk_3}$ functor, we get a simplicial Hopf algebra with Moore complex of length two\footnote{Instead of four, since the specific definitions of face maps -- as in the crossed module case. Moreover, this simplicial structure can be considered as the 2-nerve of a given 2-crossed module; and it is 3-coskeletal.}. This yields a functor between the categories $\mathrm{G_2 \colon \{X_2Hopf^{cc}\} \to \{SimpHopf_{\leq 2}^{cc}\}}$. The functorial constructions given above yield the required equivalence.
%the simplicial Hopf algebra ${H'} = \mathrm{cosk_3} (\mathfrak{H^{\mid 3}})$. When we calculate the face and degeneracy maps given above, we see that the length of $N(H')$ is two. Consequently, we obtain a functor $\mathrm{G_2 \colon X_2Hopf \to SimpHopf_{\leq 2}}$. 
\end{proof}

%	\begin{corollary}
%		The category of 2-crossed modules of Hopf algebras is equivalent to the category of simplicial Hopf algebras with Moore complex of length two.
%	\end{corollary}
	
%	\begin{rem}
%		Note that $H'$ is the 2-nerve of a given 2-crossed module and it is 3-coskeletal.
%	\end{rem}

\subsection{Conclusion}

Recalling the functors given in \ref{final1}, \ref{final2}, \ref{final3} and \ref{final5}, all fitting into the diagram:
\begin{align}\label{final4}
\xymatrix@R=40pt@C=40pt{
	\mathrm{\{SimpGrp_{\leq 2}\}} \ar[d] & \mathrm{\{SimpHopf_{\leq 2}^{cc}\}} \ar[r]^{Prim} \ar[d] \ar[l]_{Gl} & \mathrm{\{SimpLie_{\leq 2}\}} \ar[d] \\
	\mathrm{\{X_{2}Grp\}} \ar@<1ex>[u] & \mathrm{\{X_{2}Hopf^{cc}\}} \ar[l]_{Gl}  \ar[r]^{Prim} \ar@<1ex>[u] & \mathrm{\{X_{2}Lie\}}  \ar@<1ex>[u]
	}
	\end{align}
	where we refer to \citep{C2,MP2,E1} for \rev{the equivalences} between $\mathrm{\{SimpGrp_{\leq 2}\} \to \{X_2Grp\}}$, and $\mathrm{\{SimpLie_{\leq 2}\} \to \{X_2Lie\}}$. 
	
	\medskip
	
	\rev{Moreover, it is a natural question (suggested by the referee) as to whether the horizontal functors of \eqref{final4} are part of an adjunction as an induced version of the functors given in \eqref{bigone}. Related to the upper part: first of all, considering the following diagram:}
	\begin{align}\label{final6}
	\xymatrix@R=30pt@C=30pt{
		\mathrm{\D^{op}} \ar[d] & \mathrm{\D^{op}} \ar[d] \ar@<0.5ex>@{-}[l] \ar@<-0.5ex>@{-}[l] & \mathrm{\D^{op}} \ar@<0.5ex>@{-}[l] \ar@<-0.5ex>@{-}[l] \ar[d] \\
		\mathrm{\{Grp\}} 
		\ar@<-1.0ex>@{}[r]^(.22){}="a"^(.71){}="b" \ar_{\cal G^\sharp} "a";"b"="9"
		& \mathrm{\{Hopf^{cc}\}} 
		\ar@<-0.75ex>@{}[l]^(.3){}="a"^(.78){}="b" \ar_{Gl} "a";"b"="10"
		\ar@{}"9";"10"|(.2){\,}="11"
		\ar@{}"9";"10"|(.8){\,}="12"
		\ar@{}"12" ;"11"|{\top}
		\ar@<-1.0ex>@{}[r]^(.29){}="a"^(.77){}="b" \ar_{Prim} "a";"b"="13"
		& \mathrm{\{Lie\}} 
		\ar@<-0.75ex>@{}[l]^(.24){}="a"^(.71){}="b" \ar_{\cal F^\sharp} "a";"b"="14"
		\ar@{}"13";"14"|(.2){\,}="15"
		\ar@{}"13";"14"|(.8){\,}="16"
		\ar@{}"15" ;"16"|{\bot}
	}
	\end{align}
	\rev{we get the adjunctions}
	\begin{align*}
	\xymatrix@C=30pt@R=30pt{     
		\mathrm{\{SimpGrp\}}
		\ar@<-1.0ex>@{}[r]^(.26){}="a"^(.67){}="b" \ar_{\cal G^\sharp} "a";"b"="9" 
		&\mathrm{\{SimpHopf^{cc}\}}
		\ar@<-0.75ex>@{}[l]^(.33){}="a"^(.74){}="b" \ar_{Gl} "a";"b"="10"
		\ar@{}"9";"10"|(.2){\,}="11"
		\ar@{}"9";"10"|(.8){\,}="12"
		\ar@{}"12" ;"11"|{\top}  
		% 	\ar@{<-}[u]<1.3ex>^{U}_{\vdash}
		% 	\ar[u]<-1ex>_{\cal F}
		\ar@<-1.0ex>@{}[r]^(.34){}="a"^(.75){}="b" \ar_{Prim} "a";"b"="13"
		&\mathrm{\{SimpLie\}}
		\ar@<-0.75ex>@{}[l]^(.25){}="a"^(.66){}="b" \ar_{\cal F^\sharp} "a";"b"="14"
		\ar@{}"13";"14"|(.2){\,}="15"
		\ar@{}"13";"14"|(.8){\,}="16"
		\ar@{}"15" ;"16"|{\bot} 
	} .
	\end{align*} 
	\rev{However, as a consequence of the well-known adjunctions given in \eqref{main-adjunction}, we can immediately say that the left adjoint functors ${\cal G^\sharp}$ and ${\cal F^\sharp}$ preserve colimits (therefore cokernels). In this context, the behaviors of these functors at the (length of) Moore complex level is not automatically predictable - since the Moore complex construction is based on the kernels.
%	We therefore do not expect these functors to preserve the length of the Moore complex to reverse the arrows existing in the upper-part of diagram \eqref{final4}. 
	On the other hand, related to the lower part: we do not know whether the functor ${\cal G^\sharp}$ or ${\cal F^\sharp}$ preserves \mbox{2-crossed} module structures so far.  These will be subject of another study, together with the ones we mentioned below.}

	\medskip

 	\rev{Additionally, not only from the viewpoint of Lie algebras and groups, but also considering the \mbox{2-crossed} module version of the diagram (10) in \citep{JFM1}, it could be possible to have further functorial relationships between different algebraic structures. For instance, the categories of bare algebras, commutative algebras, Lie groups, etc. can be examined (in fact, 2-crossed modules of bare algebras or Lie groups have not been defined yet).}
	
	\medskip
	
	One level fewer of diagram \eqref{final4}, \rev{we obviously} obtain the crossed module version of the same diagram that yields new possible connections to the functors given in \citep{zbMATH06404301, zbMATH06725335} where the authors examine the \rev{functorial} relationships between the category of crossed modules of groups, of Lie algebras, of Leibniz algebras and of associative algebras. \edit{Another generalization of crossed modules and their relationships with simplicial objects is recently studied in \citep{BohmIII, BohmI} for the category of monoids in which the author proves that the category of crossed modules of monoids is equivalent to the category of simplicial monoids with Moore complex of length one.}
	
	\medskip
	
	All in all, we have unified the 2-crossed module notions of groups and of Lie algebras as well as their relationships with \edit{the} Moore complex, in the category of cocommutative Hopf algebras. \edit{On the other hand, as we mentioned in the introduction, there are some other algebraic models for homotopy 3-types such as crossed squares, cat$^2$-structures, quadratic modules, etc with some relationships (and also some equivalences) between them. In this context, the notion of crossed squares and their relationship between cat$^2$-Hopf algebras is an ongoing work due to \citep{CT2019}. Once its completed, \rev{one can examine the connection} between these two studies, which will lead to a functorial relationship between crossed squares and 2-crossed modules that \rev{will again} capture the group and Lie algebra cases.}

\appendix

\section{Appendix}

\subsection{Inverting through the antipode}
\label{appendix:dia-appr}

\revt{
Let $f,g \colon H' \to H$ be two Hopf algebra maps such that 
\begin{align*}
	\underset{(x)(y)}{\sum } f (x' \tn y') \,  S \big( g (x'' \tn y'') \big) =  \e (x) \e (y) 1_H \, .
\end{align*}
Then
\begin{align*}
	\begin{split}
	f (x \tn y) & = \underset{(x)(y)}{\sum } f (x' \tn y') \, \Big( S \big( g (x'' \tn y'') \big) \, g(x''' \tn y''') \Big) \\
	& = \underset{(x)(y)}{\sum } \Big( f (x' \tn y') \,  S \big( g (x'' \tn y'') \big) \Big) \, g(x''' \tn y''')  \\
	& = \underset{(x)(y)}{\sum } \e (x') \e (y') 1_H \, g(x'' \tn y'')  \\
	& = g (x \tn y).
	\end{split}
\end{align*}
Note that, this property is not only used in the proof of Theorem \ref{simp-xmod}, but also frequently in the proof of Theorem \ref{simp-2xmod}.}

\bigskip

\subsection{Proof of $(4)$}
\label{appendix:4} 

First of all, we have (for all $x,y \in NH_1$ and $k \in NH_2$):
\begin{align*}
\big( \d_1 (x) \rh y \big) \rh' k & = \big( s_0 d_1 (x) \aa y \big) \rh' k \\
& = \underset{(x)}{\sum } \, \big( s_0 d_1 (x') \, y \, S (s_0 d_1(x'')) \big) \rh' k \\
& = \underset{(x)}{\sum } \,  s_0 d_1 (x') \rh' \Big( y \rh' \big( s_0 d_1 (S(x'')) \rh' k \big) \Big) \\
& = \underset{(x)}{\sum } \,  s_1 s_0 d_1 (x') \aa \Big( s_1 (y) \aa \big( s_1 s_0 d_1 (S(x'')) \aa k \big) \Big) \\
& = \underset{(x)}{\sum } \,  s_0 (x') \aa \Big( s_1 (y) \aa \big( s_0 (S(x'')) \aa k \big) \Big) \\
& = \big( s_0 (x) \aa s_1(y) \big) \aa k \, .
\end{align*}

\bigskip
\medskip

\noindent Then the proof continues as follows (for all $x,y,z \in NH_1$):
\begin{align*} 
\{x , yz\} & = \underset{(xy)(z)}{\sum } \big( s_{1}(xy)' \aa s_{1}(z)' \big) \, S \big( s_{0}(xy)'' \aa s_{1}(z)'' \big) \\[1.5ex]
& = \underset{(x)(y)(z)}{\sum } \big(s_{1}(x') \aa s_{1}(y'z') \big) \, S \big(s_{0}(x'') \aa s_{1}(y''z'') \big) \\[1.5ex] 
& = \underset{(x)(y)(z)}{\sum } \big(s_{1}(x') \aa s_{1}(y') \big) \, \big(s_{1}(x'') \aa s_{1}(z') \big) \\
& \quad \quad \quad \quad \quad \quad \quad \quad
S \, \Big( \big(s_{0}(x''') \aa s_{1}(y'') \big) \big(s_{0}(x'''') \aa s_{1}(z'') \big) \Big) \\ \\
& = \underset{(x)(y)(z)}{\sum } \big( s_{1}(x') \aa s_{1}(y') \big) \, S \big( s_{0}(x'') \aa s_{1}(y'') \big) \\
& \quad \quad \quad \quad \quad \quad \quad \quad
\big( s_{0}(x''') \aa s_{1}(y''') \big) \, \big( s_{1}(x'''') \aa s_{1}(z') \big) \\[1.5ex]
& \quad \quad \quad \quad \quad \quad \quad \quad \quad \quad \quad \quad
\ S \Big( \big( s_{0}(x''''') \aa s_{1}(y'''') \big) \big(s_{0}(x'''''') \aa s_{1}(z'') \big) \Big) \\ \\
& =  \underset{(x)(y)(z)}{\sum } \{x' , y'\} \, \big(s_{0}(x'') \aa s_{1}(y'') \big) \, \big(s_{1}(x''') \aa s_{1}(z') \big) \\
& \quad \quad \quad \quad \quad \quad \quad \quad
\ S \Big( \big(s_{0}(x'''') \aa s_{1}(y''') \big) \, \big(s_{0}(x''''') \aa s_{1}(z'') \big) \Big) \\ \\
&  = \underset{(x)(y)(z)}{\sum } \, \{x' , y'\} \ \big(s_0(x'') \aa s_1(y'') \big) \aa \\
& \quad \quad \quad \quad \quad \quad \quad \quad 
\Big( \big(s_{1}(x''') \aa s_{1}(z') \big) \, S \big( s_{0}(x'''') \aa s_{1}(z'') \big) \Big)  \\ \\
&  = \underset{(x)(y)}{\sum } \, \{x' , y'\} \ \big(s_0(x'') \aa s_1(y'') \big) \aa \{x''' , z \}  \\[1.5ex]
&  = \underset{(x)(y)}{\sum } \, \{x' , y'\} \ (\d_1(x'') \rh y'') \rh' \{x''' , z \} \revt{\quad \because \,\, \text{by using \eqref{cond5}}.}
\end{align*}

\bigskip

\subsection{Proof of $(5)$}
\label{appendix:5}

For all $x,y,z \in NH_1$, we get:
		\begin{align*}
		& \{xy , z\} = \underset{(xy)(z)}{\sum } \big( s_{1}(xy)' \aa s_{1}(z)' \big) \, S \big( s_{0}(xy)'' \aa s_{1}(z)'' \big) \\[1.5ex]
		& = \underset{(x)(y)(z)}{\sum } \big( s_{1}(x'y') \aa s_{1}(z') \big) \, S \big( s_{0}(x''y'') \aa s_{1}(z'') \big) \\[1.5ex]
		& = \underset{(x)(y)(z)}{\sum } \big( s_{1}(x'y') \aa s_{1}(z') \big) \, \textcolor{black}{ \big( \e(x'') \e(y'')\e(z'') 1_{H_2} \big) } \, S \big( s_{0}(x'''y''') \aa s_{1}(z''') \big) \\[1.5ex]
		& = \underset{(x)(y)(z)}{\sum } \big( s_{1}(x'y') \aa s_{1}(z') \big) \, \textcolor{black}{ \Big( s_{0}(x'') \aa \big(\e(y'') \e(z'') 1_{H_2} \big) \Big) } \, S \big( s_{0}(x'''y''') \aa s_{1}(z''') \big) \\[1.5ex]
		& = \underset{(x)(y)(z)}{\sum } \big( s_{1}(x'y') \aa s_{1}(z') \big) \\[0ex]
		& \qqqq \textcolor{black}{ s_{0}(x'') \aa \Big( S \big( s_{1}(y'') \aa s_{1}(z'') \big) \, \big( s_{1}(y''') \aa s_{1}(z''') \big) \Big) }  \\[2ex]
		& \qqqq \quad \quad \quad \quad \quad S \big( s_{0}(x''''y'''') \aa s_{1}(z'''') \big) \\[1.5ex]
		& = \underset{(x)(y)(z)}{\sum } \big( s_{1}(x'y') \aa s_{1}(z') \big) \\
		& \qqqq \textcolor{black}{ s_{0}(x'') \aa \Big( S \big( s_{1}(y'') \aa s_{1}(z'') \big) \Big) \,\,  s_{0}(x''') \aa \Big( \big( s_{1}(y''') \aa s_{1}(z''') \big) \Big) } \\[1.5ex] %ok
		& \qqqq \quad \quad \quad \quad \quad S \big( s_{0}(x''''y'''') \aa s_{1}(z'''') \big) \\[1.5ex]
		& = \underset{(x)(y)(z)}{\sum } \textcolor{black}{ s_{1}(x') \aa \Big( s_{1}(y') \aa s_{1}(z') \Big)  } \,\, s_{0}(x'') \aa \Big( S \big( s_{1}(y'') \aa s_{1}(z'') \big) \Big) \\
		& \qqqq  s_{0}(x''') \aa \Big(s_{1}(y''') \aa s_{1}(z''')\Big) \,\,  S \big( s_{0}(x''''y'''') \aa s_{1}(z'''') \big) \\ \\
		& = \underset{(x)(y)(z)}{\sum } \textcolor{black}{ \Big( s_{1}(x') \aa s_{1}\big(y' \aa z'\big) \Big) \ S \Big( s_{0}(x'') \aa s_{1} \big( y'' \aa z'' \big) \Big) } \\ 
		& \qqqq \Big( s_{0}(x''') \aa \big(s_{1}(y''') \aa s_{1}(z''')\big) \Big) \, S \big( s_{0}(x''''y'''') \aa s_{1}(z'''') \big) \\ \\
		& = \underset{(x)(y)(z)}{\sum } \textcolor{black}{ \{x',y' \aa z'\} } \ s_{0}(x'') \aa \big(s_{1}(y'') \aa s_{1}(z'')\big) \,  S \big( s_{0}(x'''y''') \aa s_{1}(z''') \big) \\
		& = \underset{(x)(y)(z)}{\sum } \{x',y' \aa z'\}  \ s_{0}(x'') \aa \big(s_{1}(y'') \aa s_{1}(z'')\big) \\
		& \qqqq \textcolor{black}{ s_{0}(x''') \aa \Big( S \big( s_{0}(y''') \aa s_{1}(z''') \big) \Big) } \\[1.5ex] 
		& = \underset{(x)(y)(z)}{\sum } \{x',y' \aa z'\}  \ \textcolor{black}{ s_{0}(x'') \aa \Big( \big(s_{1}(y'') \aa s_{1}(z'')\big) \, S \big(s_{0}(y''') \aa s_{1}(z''')\big) \Big) } \\[1.5ex]
		& = \underset{(x)(y)(z)}{\sum } \{x',y' \aa z'\}  \ \textcolor{black}{ s_{0} (x'') \aa \{y'',z''\} } \\[1.5ex]
		& = \underset{(x)(y)(z)}{\sum } \{x',y' \aa z'\}  \ \textcolor{black}{ s_1 s_{0} d_1 (x'') \aa \{y'',z''\}  } \\[1.5ex]
		& = \underset{(x)(y)(z)}{\sum } \{x',y' \aa z'\} \ \textcolor{black}{ d_{1}(x'') \arho \{y'',z''\} } \revt{\quad \because \,\, \text{by using \eqref{cond5}}.}
		\end{align*}

\end{document}